\documentclass[10pt, a4paper]{amsart}

\usepackage[T1]{fontenc}
\usepackage[utf8]{inputenc}
\usepackage{amsmath, amssymb, amsfonts, amsthm}
\usepackage[margin=2.0cm]{geometry}
\usepackage[]{setspace}
\usepackage[usenames,dvipsnames]{color}
\usepackage[unicode]{hyperref}
\hypersetup{
    urlcolor=blue,
    colorlinks=true,
    linkcolor= blue,
    citecolor=blue,
}
\usepackage{bbm}
\usepackage{bold-extra}
\usepackage[all]{xy}

\usepackage{enumitem}
\setlist[itemize]{label=-}
\setlist[enumerate]{label={\bf(\roman*)}, itemsep=1ex,leftmargin=1.4cm,topsep=1ex}
\setlist[enumerate,2]{label={\bf(\alph*)}, itemsep=1ex,leftmargin=0.5cm,topsep=1ex}
\setlist[enumerate,3]{label={\bf\roman*}., itemsep=1ex,leftmargin=0.5cm,topsep=1ex}

\theoremstyle{plain}
\newtheorem{theorem}{Theorem}[section]
\newtheorem*{main_theorem}{Main Theorem}
\newtheorem{proposition}[theorem]{Proposition}
\newtheorem{lemma}[theorem]{Lemma}
\newtheorem{Conj}[theorem]{Conjecture}
\newtheorem{corollary}[theorem]{Corollary}
\newtheorem{fact}[theorem]{Fact}
\newtheorem{task}[theorem]{Task}

\theoremstyle{definition}
\newtheorem{example}[theorem]{Example}
\newtheorem{definition}[theorem]{Definition}

\theoremstyle{remark}
\newtheorem{claim}{Claim}[theorem]
\newtheorem{remark}[theorem]{Remark}

\newenvironment{claimproof}[1][\proofname]
  {%
    \proof[#1]%
  }
  {%
    \endproof%
  }

\usepackage{tikz}
\usetikzlibrary{shapes.geometric, arrows}
\tikzstyle{startstop} = [rectangle, rounded corners, minimum width=3cm, minimum height=1cm,text centered, draw=black, fill=red!30]
\tikzstyle{io} = [trapezium, rounded corners, trapezium left angle=70, trapezium right angle=110, minimum width=3cm, minimum height=1cm, text centered, draw=black, fill=blue!30]
\tikzstyle{process} = [rectangle, rounded corners, minimum width=3cm, minimum height=1cm, text centered, draw=black, fill=orange!30]
\tikzstyle{decision} = [rectangle, rounded corners, minimum width=3cm, minimum height=1cm, text centered, draw=black, fill=green!30]
\tikzstyle{arrow} = [thick,->,>=stealth]

\newcommand\val{\mathrm{val}}
\newcommand\oag{\mathrm{oag}}
\newcommand\res{\mathrm{res}}

\newcommand\ring{\mathrm{ring}}
\newcommand\Lring{\mathcal{L}_{\ring}}
\newcommand\Lval{\mathcal{L}_{\val}}
\newcommand\Loag{\mathcal{L}_{\oag}}

\providecommand{\rmDelta}{\mathrm{\Delta}}

\providecommand{\rmGamma}{\mathrm{\Gamma}}

\newcommand\IM{{\bf(Im)}}
\newcommand\SE{{\bf(SE)}}

\title[Characterizing NIP henselian fields]{\Large\rm Characterizing NIP henselian fields}
\author{Sylvy Anscombe and Franziska Jahnke}
\address{Universit\'{e} Paris Cit\'{e} and Sorbonne Universit\'{e}, CNRS, IMJ-PRG, F-75013 Paris, France}
\email{sylvy.anscombe@imj-prg.fr}
\address{Institute for Logic, Language and Computation,
Universiteit van Amsterdam,
1090 GE Amsterdam, Netherlands
}
\address{Institut f\"{u}r Mathematische Logik und Grundlagenforschung,
Universit\"at M\"{u}nster,
Einsteinstr. 62,
48149 M\"{u}nster,
Germany}
\email{franziska.jahnke@uni-muenster.de}

\begin{document}
\begin{abstract}
 In this paper, we characterize NIP henselian valued fields modulo
 the theory of their residue field, both in an algebraic and in a model-theoretic
 way.
 Assuming the conjecture that every infinite NIP field is either separably
 closed, real closed or admits a non-trivial henselian valuation, this
 allows us to obtain a characterization of all theories of NIP fields.
 \end{abstract}
\maketitle

\newcommand\maincasea{\bf(a)}
\newcommand\maincaseai{\bf(a.i)}
\newcommand\maincaseaii{\bf(a.ii)}
\newcommand\maincaseb{\bf(b)}
\newcommand\maincasebi{\bf(b.i)}
\newcommand\maincasebii{\bf(b.ii)}
\newcommand\maincasebiii{\bf(b.iii)}
\newcommand\maincasec{\bf(c)}
\newcommand\maincaseci{\bf(c.i)}
\newcommand\maincasecii{\bf(c.ii)}

\newcommand\Shelahcasea{\rm\textbf{(\textsc{A})}}
\newcommand\Shelahcaseb{\rm\textbf{(\textsc{B})}}
\newcommand\Shelahcasec{\rm\textbf{(\textsc{C})}}
\newcommand\Shelahcaseci{\rm\textbf{(\textsc{C}.i)}}
\newcommand\Shelahcasecii{\rm\textbf{(\textsc{C}.ii)}}

\makeatletter
\newcommand{\listintertext}{\@ifstar\listintertext@\listintertext@@}
\newcommand{\listintertext@}[1]{
  \hspace*{-\@totalleftmargin}#1}
\newcommand{\listintertext@@}[1]{
  \hspace{-\leftmargin}#1}
\makeatother

\newcommand\MAINSTATEMENT{
        \begin{enumerate}[leftmargin=1.4cm]
            \item[\bf(1)] $Kv$ is NIP.
            \item[\bf(2)] Either 
                \begin{align*}
                \text{\maincasea}& 
                \left\{\begin{array}{ll}
                    \text{\maincaseai} & \text{$(K,v)$ is of equal characteristic, and}\\
                    \text{\maincaseaii} & \text{$(K,v)$ is trivial or separably defectless Kaplansky;}
                \end{array}\right.
                \intertext{or}
                \text{\maincaseb}&
                \left\{\begin{array}{ll}
                    \text{\maincasebi} & \text{$(K,v)$ has mixed characteristic $(0,p)$, and}\\
                    \text{\maincasebii} & \text{$(K,v_{p})$ is finitely ramified, and}\\
                    \text{\maincasebiii} & \text{$(Kv_{p},\bar{v})$ is trivial or separably defectless Kaplansky;}
                \end{array}\right.
                \intertext{or}
                \text{\maincasec}&
                \left\{\begin{array}{ll}
                    \text{\maincaseci} & \text{$(K,v)$ has mixed characteristic $(0,p)$, and}\\
                    \text{\maincasecii} & \text{$(Kv_{0},\bar{v})$ is defectless Kaplansky.}
                \end{array}\right.
                \end{align*}
        \end{enumerate}
}

\section{Introduction}
Since Macintyre showed in the early seventies that infinite
$\omega$-stable fields are algebraically closed (\cite{Mac71}), 
the question of
whether key model-theoretic tameness properties coming from 
Shelah's classification theory (like stability, simplicity, NIP) 
correspond to natural algebraic definitions when applied to fields
has been studied extensively.
The most prominent instance is the Stable Fields Conjecture, predicting that 
any infinite stable field is separably closed. In 1980, 
Cherlin and Shelah generalized Macintyre's result to superstable fields (\cite{CS80}),
but despite much effort, no further progress was made for a long time.
Very recently, the Stable Fields Conjecture was solved in the special case of \emph{large}
stable fields by Johnson, Tran, Walsberg, and Ye (\cite{JTWY}), using the newly-introduced 
\'etale-open topology.
Nevertheless, the Stable Fields Conjecture in full generality still seems to be far beyond our reach.
Its generalization to NIP fields has received much attention:
\begin{Conj}[Conjecture on NIP fields] \label{sh}
	Let $K$ be an infinite NIP field. Then $K$ is separably closed, real closed
	or admits a non-trivial henselian valuation.
\end{Conj}
Conjecture \ref{sh} has many variants but no clear origin, and is usually
attributed to Shelah (who stated a closely related conjecture on 
strongly NIP fields
and asked for a similar description of NIP fields in \cite{Sh}).
The special case of fields of finite dp-rank was recently proven by Johnson
in a series of spectacular preprints culminating in \cite{Joh2}.

Apart from separably closed fields, real closed fields, and the $p$-adics plus
their finite extensions, the only currently known method to construct NIP
fields is by NIP transfer theorems in the spirit of Ax-Kochen/Ershov:
under certain algebraic assumptions, if $(K,v)$
is a henselian valued field such that the residue field $Kv$ is NIP, then
$(K,v)$ is NIP. The first such theorem was shown by Delon:
\begin{fact}[Delon, \cite{Del81}]
	Let $(K,v)$ be a henselian valued field of equicharacteristic $0$.
	Then,
	$$(K,v) \textrm{ is NIP in }\Lval \Longleftrightarrow
	Kv \textrm{ is NIP in }\Lring.$$
\end{fact}
Here, $\Lring = \{0,1,+,\cdot\}$ denotes the language of rings and 
$\Lval$ is a three-sorted language with sorts for the field, the residue
field and the value group (see the beginning of section \ref{basic} for a precise
definition of $\Lval$).

Note that Delon originally proved the theorem under the additional assumption
that the value group $vK$ is NIP as an ordered abelian group. It was later
shown by Gurevich and Schmidt that this holds for any ordered abelian group (\cite[Theorem 3.1]{GS}).
Several variants of Delon's theorem were proven in mixed and positive
characteristic, first by B\'elair (\cite{Bel99}) and more recently by Jahnke and Simon (\cite{JS19}).
B\'elair showed that an algebraically maximal Kaplansky field $(K,v)$
of positive characteristic is NIP in $\Lval$ if and only if its residue
field $Kv$ is NIP in $\Lring$, and that the same holds
if $(K,v)$ is finitely ramified with perfect residue field.
Jahnke and Simon generalized B\'elair's result to separably
algebraically maximal Kaplansky fields of finite degree of imperfection and
arbitrary characteristic. Conversely, they use the theorem by 
Kaplan, Scanlon and Wagner stating that NIP fields of positive characteristic
admit no Artin-Schreier extensions (\cite{KSW}) 
to show that NIP henselian valued fields
of positive characteristic are separably algebraically maximal.
The approach used by Jahnke and Simon builds on machinery developed by
Chernikov and Hils in the NTP$_2$ context (\cite{CH14}).
Also following this route, we prove what one might consider as the ultimate 
transfer theorem:
our main result is that a henselian valued field $(K,v)$
is NIP (in $\Lval$) if and only if its residue field $Kv$ is NIP (in $\Lring$)
and the valued field satisfies a list of purely algebraic conditions
(all of which are preserved under $\Lval$-elementary equivalence).
More precisely, we show the following

\begin{main_theorem}[{Theorem \ref{thm:1}}]\label{thm:main_theorem}
Let $(K,v)$ be a henselian valued field.
Then $(K,v)$ is NIP in $\Lval$ 
	if and only if both of the following hold:

\MAINSTATEMENT
\end{main_theorem}
Here, for a valuation $v$ of mixed characteristic $(0,p)$, 
we use $v_0$ to denote the finest coarsening of $v$ with residue characteristic $0$ 
and $v_p$ to denote the coarsest coarsening of $v$ with residue characteristic $p$
(see Remark~\ref{rem:sd}).

We then apply our main theorem in two different ways. In Corollary \ref{cor:3},
we show that the henselization $(K^h,v^h)$ of any NIP valued field $(K,v)$ is
again NIP. Moreover, we give a classification of NIP fields assuming that 
Conjecture \ref{sh} holds (see Theorem \ref{thm:3}).

We now give an overview over the contents of this paper.
In section
\ref{basic}, we start by introducing the valuation-theoretic notions
which appear in the main result,
then we prove several lemmas about valued fields
which are applied later in the paper. We also survey Ax-Kochen/Ershov principles, which are fundamental in the model theory of valued fields.

Next, in section \ref{sec:NIP}, we recall
the definition of NIP and prove some facts about NIP valued fields
(without assuming henselianity). These allow us to prove Theorem \ref{thm:2}
which states that any NIP valued field (not assumed to be henselian) 
satisfies both of the properties
{\bf (1)} and {\bf (2)} occurring in the Main Theorem. In particular, Theorem
\ref{thm:2}
entails the left-to-right
implication of the Main Theorem.

Section \ref{sec:trans} recalls known results about henselian NIP fields and 
contains two new NIP transfer results:
In Proposition \ref{prp:SAMK}, we show that any separably algebraically maximal Kaplansky
valued field of positive characteristic is NIP in $\Lval$ if and only if
its residue field is NIP in $\Lring$. This was previously known under the
additional assumption of finite degree of imperfection (\cite[Lemma 3.2]{JS19}). 
As a
consequence, we obtain a new proof of an unpublished result of Delon 
that the $\Lval$-theory of any separably closed valued field is NIP (Corollary
\ref{cor:SCVF}).
The second NIP transfer result contained in this section
is concerned with finitely ramified valued fields,
i.e., valued fields $(K,v)$ of mixed characteristic $(0,p)$ such that 
the interval $(0,v(p)]$ in the value group $vK$ is finite.
\label{def:fin_ram}
We show that any henselian finitely ramified valued field
with NIP residue field is NIP in $\Lval$ in
Corollary \ref{cor:fin_ram_2}.
Moreover, this also holds if we compose a henselian finitely ramified valuation with a henselian NIP valuation on the residue field (Proposition \ref{prp:fin_ram}).
These results
generalize a theorem by B\'elair stating that a henselian unramified valued field with perfect NIP residue field is NIP in $\Lval$ (\cite[Th\'{e}or\`{e}me 7.4(2)]{Bel99}).
The key ingredient is a new understanding of the model theory of
henselian unramified valued fields (where the residue field is allowed to be imperfect)
by the authors, cf.~\cite{AJ2}.

In section \ref{sec:mt}, we state and prove our Main Theorem (Theorem \ref{thm:1}).
We also give a number of examples showing that none of the clauses in the theorem is vacuous.
As a consequence of Theorem \ref{thm:1}, we show that the
henselization of any NIP valued field is again NIP (Corollary \ref{cor:3}).
This gives a NIP analogue to a result by Hasson and Halevi who showed that the henselization
of every strongly NIP (also known as strongly dependent) valued field is again 
strongly NIP (\cite{HH2}).\footnote{Note that there are many examples of NIP valued fields which 
are not strongly NIP since all strongly NIP fields are perfect.
For explicit examples, see \ref{Ex:NIP}.}

The last two sections contain variants of our Main Theorem,
although not straightforward ones.
The version presented in section \ref{sec:mtv}
has a distinctly more model-theoretic flavour:
throughout the section, given
a complete $\Lring$-theory $T_k$ of NIP fields, we describe all the
complete $\Lval$-theories of NIP henselian valued fields $(K,v)$ such that the residue field is a model of $T_k$.
This is rather easy
and unsurprising in equal characteristic 
(and even well-known in equicharacteristic $0$),
but noticeably harder in mixed characteristic.
First, we introduce some theories of henselian valued fields
of mixed characteristic and prove their completeness. 
After taking a closer look at finitely ramified fields 
(cf.~Lemma~\ref{lem:finite_ramification}),
we give the desired
characterization in Proposition \ref{prp:elem}.

In the final section, we apply this characterization to give a refinement of Conjecture \ref{sh}:
we give a list of complete $\Lval$-theories of henselian valued fields which are all NIP in $\Lval$, 
and we show that Conjecture
\ref{sh} implies that every NIP field $K$ admits a henselian valuation $v$ such that 
$(K,v)$ is a model of one of the theories listed (see Theorem \ref{thm:3}). 
This is a NIP analogue of 
a similar conjectural classification of strongly NIP
fields which was obtained by Halevi, Hasson and Jahnke (\cite{HHJ2}).

\section{Basic notions from valuation theory}
\label{basic}

In this section, we introduce the valuation-theoretic notions  
which appear in our main result (Theorem \ref{thm:1}). 
Furthermore, we prove 
a number of valuation-theoretic lemmas which will be applied in later sections.

Throughout this paper, we denote the valuation ring of a valued field $(K,v)$ by $\mathcal{O}_{v}$, the valuation ideal by $\mathfrak{m}_{v}$.
When considering valued fields
we use the three-sorted language $\Lval$ with sorts $\mathbb{K}$ for the field, $\mathbbm{k}$ for the residue field and $\mathbb{G}$ for the value group
together with an additional element for $\infty$.
On both $\mathbb{K}$ and $\mathbbm{k}$, we have the language of rings $\Lring=\{0,1,+,\cdot\}$,
and on $\mathbb{G}$ the language $\mathcal{L}_{\text{oag}} = \{0,+,<\}$ of ordered
abelian groups together with a constant symbol for $\infty$, all interpreted in the usual way.
Furthermore, there are two function symbols connecting the sorts to one another, namely
$\underline{v}:\mathbb{K}\to \mathbb{G}$
and
$\underline{\text{res}}:\mathbb{K} \to \mathbbm{k}$.
Whenever we consider a valued field $(K,v)$ as a first-order structure, we
mean the corresponding $\Lval$-structure
given by the field $K$ in the sort $\mathbb{K}$,
the residue field $Kv$ in the sort $\mathbbm{k}$,
and
the value group $vK$ with its additional element $\infty$ in the sort $\mathbb{G}$.
Naturally, $\underline{v}$ is interpreted by the valuation $v$ and $\underline{\text{res}}$ is interpreted by the residue map $\res:\mathcal{O}_{v}\to Kv$ which we extend by setting
$\text{res}(x)=0$
for $x \in K$ with $v(x)<0$.

Note that we choose this language for 
convenience since it allows us to refer to the residue field and the value group
as objects in our language. Other commonly used languages of valued fields include
the one-sorted language $\Lring \cup \{\mathcal{O}\}$, the expansion of the language
of rings by a predicate for the valuation ring, and the two-sorted language
with sorts $\mathbb{K}$ and $\mathbb{G}$ together with a map $\underline{v}:\mathbb{K}\to \mathbb{G}$, where the $\mathbb{K}$-sort is again endowed with $\Lring$ and the $\mathbb{G}$-sort
with $\mathcal{L}_\textrm{oag} \cup \{\infty\}$.
If we consider a valued field $(K,v)$ in any of these
three languages, it is biinterpretable with each of the corresponding two structures
in the other languages.
Thus, as our focus is on the question of whether the
valued field is NIP (and this is preserved under interpretability), the answer is independent from the language 
we choose.

\begin{definition}\label{def:Kaplansky}
A valued field $(K,v)$
is {\bf Kaplansky}
if 
either it is of equal characteristic zero,
or if it is of residue characteristic
$p>0$
and
\begin{enumerate}
\item
$vK$ is $p$-divisible,
\item
$Kv$ is perfect, and
\item
$Kv$ admits no proper separable algebraic extensions of degree divisible by $p$.
\end{enumerate}
\end{definition}

\begin{remark}
Equivalently,
a valued field $(K,v)$ of residue characteristic $p>0$ is
Kaplansky
if and only if
$vK$ is $p$-divisible and
$Kv$ admits no proper finite extensions of degree divisible by $p$.
\end{remark}

We now introduce both the notions of defectlessness and separable 
defectlessness. Note that if a valued field $(K,v)$ is defectless, it is always 
separably defectless;
the converse holds if we assume $K$ to be perfect, 
but does not hold in general.

\begin{definition}\label{def:defectless}
A valued field $(K,v)$ is {\bf (separably) defectless} if,
whenever $L/K$ is a 
finite (separable) field extension, 
the fundamental equality holds:
\begin{align*}
	[L:K] &= \sum_{w\supseteq v}e(w/v)f(w/v),
\end{align*}
where $w$ ranges over all prolongations of $v$ to $L$, $e(w/v)=(wL:vK)$ is the ramification degree, and $f(w/v)=[Lw:Kv]$ is the inertia degree of the extension $(L,w)/(K,v)$.
\end{definition}
 
Note that defect can ony occur in positive
residue characteristic (cf.~\cite[Theorem 3.3.3]{EP}) The next lemma shows that defectlessness is an $\Lval$-elementary property.
\begin{lemma}\label{lem:defect_ultrapower}
There is an $\Lval$-theory $T_{d}$ which axiomatizes the class of defectless valued fields.
\end{lemma}
\begin{proof}
Note that a valued field is defectless if and only if the fundamental inequality holds for all finite normal extensions.
For convenience, we fix a valued field $(K,v)$, and let $n\in\mathbb{N}$.
First, there is a standard method to uniformly interpret in $K$ the family of normal extensions $L/K$ of degree $n$:
one quantifies over tuples which form the coefficients of the minimal polynomial of a generator of such an extension.
We fix one such tuple $(c_{0},\dots,c_{n-1})$, corresponding to a normal extension $L/K$ of degree $n$.
We view the whole as an $\Lval^{1}$-structure $(K,L,v)$, where $\Lval^{1}$ is the expansion of $\Lval$ by an additional sort $\mathbb{L}$ equipped with $\Lring$ and interpreted by $L$, as well as a distinguished embedding of $K$ into $L$.
Next we show that in $(K,L,v)$ the family of valuation rings on $L$ corresponding to prolongations of $v$ is definable using the parameters $(c_{0},\dots,c_{n-1})$.
For this we use an argument of Johnson, specifically the proof of \cite[Lemma 9.4.8]{Joh1}.
The only adjustment we need to make to Johnson's argument is that in our case $\mathcal{O}_{v}$ is definable (not only $\vee$-definable), and so the second condition in \cite[Claim 9.4.9]{Joh1} is definable (not only type-definable).
The rest of the argument goes through verbatim, and it follows that $\mathcal{O}_{w}$ is definable.
More precisely, there are parameters $b_{1},\dots,b_{m}\in L$ and a formula
$\pi(x,y_{1},\dots,y_{m},z_{0},\dots,z_{n-1})$
such that
$\pi(x,b_{1},\dots,b_{m},c_{1},\dots,c_{n-1})$
defines in $(K,L,v)$ the valuation ring $\mathcal{O}_{w}$.
Therefore 
$$
	\{\mathcal{O}_{w}\;|\;\text{$w$ prolongs $v$}\}
$$ 
is a definable family in $(K,L,v)$ using parameters $(c_{0},\dots,c_{n})$.
Finally, it is clear that there is an $\Lval^{1}$-theory which axiomatizes the class of those $(K,L,v)$ which satisfy the fundamental equality.
Combining these steps, we are done.
\end{proof}

Closely related to defectlessness is the following notion:

\begin{definition}\label{def:SAM}
A valued field $(K,v)$ is {\bf (separably) algebraically maximal} 
if it admits no proper (separable) algebraic immediate extensions.
\end{definition}

We now explain how (separable) defectlessness is connected to (separable)
algebraic maximality. If $(K,v)$ is a henselian valued (separably)
defectless field, then $(K,v)$ is already (separably) algebraically maximal. 
The converse implication fails in general, but holds in henselian
NIP valued fields.%
\footnote{By Theorem \ref{thm:2}, NIP valued fields are compositions of finitely ramified and Kaplansky valued fields, and for such fields algebraic maximality implies henselian defectlessness.
For example see \cite[Theorem 3.2]{Kuh16}.}

In the cases {\maincaseb} and {\maincasec} of Theorem \ref{thm:1}, we 
decompose a mixed characteristic valuation into two equicharacteristic
valuations and a rank-1 valuation of mixed characteristic. 
This is a standard trick 
for which we explain notation and give details below.

\begin{remark}\label{rem:sd}
Let $(K,v)$ be a valued field of mixed characteristic $(0,p)$.
First, let $\rmDelta_{p}$ denote the maximal convex subgroup of $vK$ that does not contain $v(p)$,
and let $\rmDelta_{0}$ denote the minimal convex subgroup of $vK$ that does contain $v(p)$.
So we have a chain $\rmDelta_{p}<\rmDelta_{0}\leq vK$.
Next, let $v_{p}$ be the coarsening of $v$ corresponding to $\rmDelta_{p}$, and let $v_{0}$ be the coarsening of $v$ corresponding to $\rmDelta_{0}$. We use $\bar{v}$ to denote the valuation
induced by $v$ on each of the residue fields $Kv_p$ and $Kv_0$ of the coarsenings of $v$,
and $\bar{v}_p$ to denote the valuation induced by $v_p$ on the residue field
of its coarsening $Kv_0$.
In particular, $(K,v_0)$ and $(Kv_p,\bar{v})$ are valued fields
of equicharacteristic $0$ and $p$ respectively, and $(Kv_0,\bar{v}_p)$
is a rank-1 valued field of mixed characteristic
with value group
$\Delta_{0}/\Delta_{p}$.
We will make repeated use of this decomposition, which we call the
{\bf standard decomposition}.
\end{remark}

It is illustrated by the following picture, in which the arrows represent the places corresponding to the valuations rather than the valuations themselves.

\newcommand\Kstandarddecomposition{
$$
\renewcommand{\labelstyle}{\textstyle}
\xymatrix@=4em{
K\ar@{->}[r]^{vK/\rmDelta_{0}}
&
Kv_{0}\ar@{->}[r]^{\rmDelta_{0}/\rmDelta_{p}}
&
Kv_{p}\ar@{->}[r]^{\rmDelta_{p}}
&
Kv
}
$$
}
\newcommand\Kstarstandarddecomposition{
$$
\renewcommand{\labelstyle}{\textstyle}
\xymatrix@=4em{
K^{*}\ar@{->}[r]^{v^{*}K^{*}/\rmDelta_{0}^{*}}
&
K^{*}v_{0}^{*}\ar@{->}[r]^{\rmDelta_{0}^{*}/\rmDelta_{p}^{*}}
&
K^{*}v_{p}^{*}\ar@{->}[r]^{\rmDelta_{p}^{*}}
&
K^{*}v^{*}
}
$$
}
\newcommand\Kstarstandarddecompositionfr{
$$
\renewcommand{\labelstyle}{\textstyle}
\xymatrix@=4em{
K^{*}\ar@{->}[r]^{v^{*}K^{*}/\rmDelta_{0}^{*}}
&
K^{*}v_{0}^{*}\ar@{->}[r]^{\rmDelta_{0}^{*}}&
K^{*}v^{*}
}
$$
}
\newcommand\Khstandarddecomposition{
$$
\renewcommand{\labelstyle}{\textstyle}
\xymatrix@=4em{
K^{h}\ar@{->}[r]^{v^{h}K^{h}/\rmDelta_{0}^{h}}
&
K^{h}v_{0}^{h}\ar@{->}[r]^{\rmDelta_{0}^{h}/\rmDelta_{p}^{h}}
&
K^{h}v_{p}^{h}\ar@{->}[r]^{\rmDelta_{p}^{h}}
&
K^{h}v^{h}
}
$$
}
\newcommand\Kcomposition{
$$
\renewcommand{\labelstyle}{\textstyle}
\xymatrix@=4em{
K\ar@{->}[r]^{v^{0}}\ar@/_1pc/[rr]_{v}
&
Kv^{0}\ar@{->}[r]^{\bar{v}}
&
Kv
}
$$
}
\newcommand\KLcomposition{
$$
\renewcommand{\labelstyle}{\textstyle}
\xymatrix@=4em{
L\ar@{->}[r]_{u_{0}}\ar@/^1pc/[rr]^{u}
&
Lu_{0}\ar@{->}[r]_{\bar{u}}
&
Lu
\\
K\ar@{->}[r]^{v_{0}}\ar@/_1pc/[rr]_{v}\ar@{-}[u]
&
Kv_{0}\ar@{->}[r]^{\bar{v}}\ar@{-}[u]
&
Kv\ar@{-}[u]
}
$$
}
\begin{figure}[ht]
\Kstandarddecomposition
\caption{The Standard Decomposition}
\label{fig:1}
\end{figure}

\begin{lemma}\label{lem:discrete}
For valued fields of mixed characteristic $(0,p)$,
the case distinction
\begin{enumerate}
\item $\rmDelta_{0}/\rmDelta_{p}$ is discrete
\item $\rmDelta_{0}/\rmDelta_{p}$ is not discrete
\end{enumerate}
is preserved under $\Lval$-elementary equivalence.
\end{lemma}
\begin{proof}
Let $(K,v)$ be a valued field of mixed characteristic $(0,p)$, viewed according to the standard decomposition.
If $\rmDelta_{0}/\rmDelta_{p}$ is not discrete,
then for each $n>0$ there exists $N_{n}>n$ and $x\in K$ such that
\begin{align*}
	0<nv(x)\leq vp\leq N_{n}v(x),
\end{align*}
	i.e.~$v(x)$ is in the interval 
	$\big[\frac{v(p)}{N_{n}},\frac{v(p)}{n}\big]\subseteq vK$.
The existence of such an element $x$ is expressed by a sentence 
$\varphi_{n,N_{n}}$ in the language of valued fields.
Indeed, $\rmDelta_{0}/\rmDelta_{p}$ is not discrete if and only if
$(K,v)\models\bigwedge_{n>0}\varphi_{n,N_{n}}$
for some function $n\longmapsto N_{n}$ such that $N_{n}>n$, for all $n$.
Therefore {\bf(i)} (and hence also {\bf(ii)}) is preserved under $\Lring$-elementary equivalence.
\end{proof}

For lack of an appropriate reference, we state and prove the following lemmas,
which ensure that several of the properties we are interested in behave well
under compositions of valuations. These will come in particularly handy when
we use the standard decomposition to study mixed characteristic valued fields.

\begin{lemma}\label{lem:henselization}
Let $(K,v)$ be a valued field such that $v$ is equal 
to a composition $\bar{v}\circ v^{0}$ of valuations,
i.e., the place corresponding to $v$ can be decomposed into two places as depicted in the following
diagram:
\Kcomposition
\begin{enumerate}
\item Assume that $L/K$ is an algebraic extension of fields, and let 
$w$ be a prolongation of $v$ to $L$.
Then, there is a unique prolongation $w^{0}$ of $v^{0}$ to $L$ 
which is a coarsening of $w$.
\item
Let $(K^{h},v^{h})$ be the henselization of $(K,v)$. Then 
$(K^{h}w^0,\bar{v}^{h})$ is the henselization of $(Kv^0,\bar{v})$, where $w^0$
denotes the unique coarsening of $v^h$ prolonging $v^0$.
\end{enumerate}
\end{lemma}
\begin{proof}
\begin{enumerate}
\item
The valuation $v^{0}$ corresponds to a convex subgroup $\rmDelta$ of $vK$, and likewise
a coarsening $w^{0}$ of $w$ corresponds to a convex subgroup $\rmDelta'\leq wL$.
Such a $w^{0}$ is a prolongation of $v^{0}$ if and only if $\rmDelta'\cap vK=\rmDelta$.
Since $L/K$ is algebraic, $wL/vK$ is torsion, and the only possible choice for $\rmDelta'$ is the convex hull of $\rmDelta$ in $wL$.
\item
Since $K^h/K$ is algebraic, there is
a unique coarsening $w^0$ of $v^{h}$ which prolongs $v^0$.
Since a composition of two valuations is henselian if and only if
both components are henselian (\cite[Corollary 4.1.4]{EP}), 
$(K^{h}w^0,\bar{v}^{h})$ is henselian.
Let $(L,\bar{u})$ be an extension of $(Kv^0,\bar{v})$ which is henselian.
There exists an extension $(M,\hat{w})/(K,v^0)$ which is henselian and has residue field $M\hat{w}=L$.
Then the composition $u:=\bar{u}\circ\hat{w}$ is a henselian valuation on $M$ which prolongs $v$.
Thus $(K^{h},v^{h})$ may identified with a valued subfield of $(M,u)$.
The restriction of $\hat{w}$ to $K^{h}$ coincides with $w^0$, and induces an embedding $(K^{h}w^0,\bar{v}^{h})\subseteq(L,\hat{u})$,
which shows that $(K^{h}w^0,\bar{v}^{h})$ satisfies the universal property 
of the henselization of $(Kv^0,\bar{v})$.
\qedhere
\end{enumerate}
\end{proof}

\begin{lemma}\label{lem:defectless}
Let $(K,v)$ be a valued field such that $v$ is equal to a composition $\bar{v}\circ v^{0}$ of valuations.
Then $(K,v)$ is defectless if and only if both $(K,v^{0})$ and $(Kv^{0},\bar{v})$ are defectless.
\end{lemma}
\begin{proof}
For a finite extension $L/K$ of fields,
let $u^{0}_{1},\ldots,u^{0}_{s}$ be the distinct prolongations of $v^{0}$ to $L$.
For each $i\in\{1,\ldots,s\}$, let $\bar{u}_{i,1},\ldots,\bar{u}_{i,r_{i}}$ be the distinct prolongations of $\bar{v}$ to $Lu^{0}_{i}$,
and write $u_{i,j}:=\bar{u}_{i,j}\circ u_{i}^{0}$.
It follows from Lemma \ref{lem:henselization}{\bf(i)} that $(u_{i,j}:i\leq s,j\leq r_{i})$ enumerates the set of prolongations of $v$ to $L$.
Applying the Fundamental Inequality (\cite[Theorem 3.3.4]{EP}) several times, we have the following calculation:
\begin{align}
[L:K]
    &\geq \sum_{i\leq s}e(u^{0}_{i}/v^{0})\,f(u^{0}_{i}/v^{0}) \label{eq:1} \tag{1}\\
    &\geq \sum_{i\leq s}e(u^{0}_{i}/v^{0})\,\sum_{j\leq r_{i}}e(\bar{u}_{i,j}/\bar{v})\,f(\bar{u}_{i,j}/\bar{v}) \label{eq:2} \tag{2} \\
    &= \sum_{i\leq s}\sum_{j\leq r_{i}}e(u_{i,j}/v)\,f(u_{i,j}/v), \tag*{}
\end{align}
since $e(u_{i,j}/v)=e(u_{i}^{0}/v^{0})e(\bar{u}_{i,j}/\bar{v})$ and $f(u_{i,j}/v)=f(\bar{u}_{i,j}/\bar{v})$.
If $(Kv^{0},\bar{v})$ is defectless
then
\begin{align}
	f(u_{i}^{0}/v^{0}) &= \sum_{j\leq r_{i}}e(\bar{u}_{i,j}/\bar{v})\,f(\bar{u}_{i,j}/\bar{v}), \label{eq:4} \tag{3}
\end{align}
for each $i$.
If both $(K,v^{0})$ and $(Kv^{0},\bar{v})$ are defectless,
then the inequalities in (\ref{eq:1}) and (\ref{eq:2}) are equalities,
which verifies that $(K,v)$ is defectless.
On the other hand, if $(K,v)$ is defectless,
then we have
\begin{align*}
	[L:K] &= \sum_{i\leq s}\sum_{j\leq r_{i}}e(u_{i,j}/v)\,f(u_{i,j}/v).
\end{align*}
It follows that the inequalities in (\ref{eq:1}) and (\ref{eq:2}) are equalities again.
The first of these equalities verfies that $(K,v^{0})$ is defectless,
and the second equality implies the equations (\ref{eq:4}), for each $i$.

Continuing to assume that $(K,v)$ is defectless,
it remains to verify that $(Kv^{0},\bar{v})$ is defectless, for which we consider an arbitrary finite extension $E/Kv^{0}$.
For example by \cite[Theorem 2.14]{Kuh04},
we may take a finite extension $(F,w^{0})/(K,v^{0})$ such that $Fw^{0}/Kv^{0}$ is isomorphic to $E/Kv^{0}$
and
\begin{align*}
    [F:K] &= e(w^{0}/v^{0})\,f(w^{0}/v^{0}).
\end{align*}
By identifying $L$ with $F$, and $w^{0}$ with $u_{1}^{0}$, 
we are again in the situation considered above (where now $s=1$).
We have already shown that there is equality in (\ref{eq:4}),
which verifies that $(Kv^{0},\bar{v})$ is defectless.
\end{proof}

The property of separable defectlessness does not behave under composition in the same way,
nor does the property of being henselian (and) separably defectless. In order to give an example, we introduce the standard construction of a Cohen ring over an imperfect field. Cohen rings and their
quotient fields occur at several points throughout this paper.

A \emph{Cohen ring} (see e.g.~\cite{Coh}) is a complete Noetherian local ring $A$ with maximal ideal $pA$, where $p$ is the residue characteristic of $A$.
Such a ring is {\em strict} if it is an integral domain.
For each field $k$ of characteristic $p>0$ there exists a strict Cohen ring $C[k]$ with residue field $k$, unique up to isomorphism.
Its quotient field 
then admits a complete unramified henselian valuation $v$ with 
valuation ring $C[k]$, value group $\mathbb{Z}$ 
and residue field $k$. We denote it by $C(k)$ and call it a Cohen field over $k$.
When $k$ is perfect, the Witt ring $W[k]$ and the Cohen ring $C[k]$ coincide.
Note that $C[k]$ is unique up to isomorphism, but 
-- when $k$ is imperfect -- it is not unique up to
unique isomorphism.
For a recent treatment of the algebra and model theory of Cohen rings, see \cite{AJ2}.

\begin{remark}
Let $(k,u)$ be a separably closed valued field of characteristic $p>0$ of imperfection degree\footnote{Here we adopt the convention for a field $K$ that $p^{e}=[K:K^{(p)}]$ unless this degree is infinite, in which case we simply write $e=\infty$.}
$e>0$,
so that $k$ is imperfect.
Let $(C(k),v)$ be a Cohen field over $k$.
Since $k$ is separably closed, in particular $(k,u)$ is henselian separably defectless.
Moreover, $(C(k),v)$ is maximal, thus is even henselian defectless.
Since $k$ is imperfect, it admits a proper purely inseparable extension $k'/k$, to which $u$ extends uniquely to a valuation $u'$.
Then $(k',u')/(k,u)$ is a proper immediate extension.
Therefore $(C(k'),u'\circ v')/(C(k),u\circ v)$ is a proper immediate extension of valued fields in characteristic $0$, which in particular is separable.
This shows that $(C(k),u\circ v)$ is not separably defectless, despite $u\circ v$ being the composition of two henselian separably defectless valuations.
\end{remark}

A generalization of algebraically 
maximal Kaplansky fields is given by
tame\footnote{Tameness here is a purely algebraic notion of valued fields, not to be confused with model-theoretic tameness notions like NIP.} valued fields.
An algebraic extension $(L,w)/(K,v)$ of valued fields is
{\em tame}
if
\begin{enumerate}
\item
$(wL:vK)$ is coprime to the residue characteristic,
\item
$Lw/Kv$ is separable, and 
\item
$(L,w)/(K,v)$ is defectless.
\end{enumerate}
A valued field $(K,v)$ is
{\em tame}
if all algebraic
extensions are tame.
For more details and equivalents definitions, see
\cite{Kuh16}.
We encounter tame valued fields in the final two sections of this paper.

\subsection*{Ax-Kochen/Ershov principles}
A fundamental principle in the model theory of valued fields is that in ``well-behaved'' valued 
fields, the model theory of the valued field should be determined by the model theory of its 
residue field and value group. This goes back to the
seminal work by Ax and Kochen and, independently,
Ershov, on the model theory of the $p$-adic numbers.
Our main theorem implies in particular that for all henselian NIP valued fields $(K,v)$, the valuation be decomposed
into finitely many pieces, all of which fit into some
Ax-Kochen/Ershov setting. 
There is one Ax-Kochen/Ershov principle for relative completeness 
and one for relative model completeness, both of which occur in this paper. 
Let $\mathcal{K}$ be an class of valued fields. We say that $\mathcal{K}$ satisfies 
$\mathrm{AKE}^{\equiv}$
if for all $(K,v)$, $(L,w) \in \mathcal{K}$, we have
\begin{align*}
    \underbrace{(K,v) \equiv (L,w)}_{\text{in }\mathcal{L}_\mathrm{val}}
    \Longleftrightarrow
    \underbrace{vK\equiv wL}_{\text{in }\mathcal{L}_{\mathrm{oag}}}
    \text{ and }
    \underbrace{Kv\equiv Lw}_{\text{in }\mathcal{L}_{\mathrm{ring}}}.
\end{align*}

This principle holds in case $\mathcal{K}$ is the class of henselian fields of equicharacteristic $0$
(\cite{AxKochen-I,Er65}). In equal positive characteristic $p$, it holds for the class of separably algebraically maximal Kaplansky valued fields of characteristic $p$ and fixed degree of imperfection
$e$ (\cite{Del82}),
and for
the class of tame valued fields of characteristic $p$ (\cite{Kuh16}). In mixed characteristic, it holds for unramified
henselian valued fields (\cite{Bel99}, \cite{ErshovMult}, \cite{AJ2}).

The second such principle gives criteria for an embedding to be elementary. We say that  $\mathcal{K}$ satisfies 
$\mathrm{AKE}^{\preceq}$
if for all $(K,v)$, $(L,w)\in\mathcal{K}$,
with $(K,v)\subseteq(L,w)$,
we have
\begin{align*}
    \underbrace{(K,v) \preceq (L,w)}_{\text{in }\mathcal{L}_\mathrm{val}}
    \Longleftrightarrow
    \underbrace{vK\preceq wL}_{\text{in }\mathcal{L}_{\mathrm{oag}}}
    \text{ and }
    \underbrace{Kv\preceq Lw}_{\text{in }\mathcal{L}_{\mathrm{ring}}}.
\end{align*}
This principle also holds for all the cases mentioned above, and moreover in tame valued fields
of mixed characteristic (\cite{Kuh16}) and finitely ramified henselian valued fields
(\cite{ErshovMult, ZieglerDiss}).

\section{NIP valued fields}
\label{sec:NIP}
Let $T$ be a complete $\mathcal{L}$-theory. 
Recall that a formula $\varphi(\bar{x},\bar{y})$ has the independence
property (IP) if there is a model $M \models T$ and sequences $(\bar{a}_i)_{i\in \mathbb{N}}$ in $M^{|\bar{x}|}$
and $(\bar{b}_J)_{J \subseteq \mathbb{N}}$
in $M^{|\bar{y}|}$ such that we have
$$M \models \varphi(\bar{a}_i,\bar{b}_J) \Longleftrightarrow i\in J.$$
If there is some formula with IP, we say that $T$ has IP. Otherwise, 
we say that $T$ has NIP. For an introduction to NIP theories,
see \cite{Simon}.

Throughout this paper, we are interested in NIP fields and
NIP valued fields. We call a field $K$ (respectively, a valued field $(K,v)$) NIP 
if its $\Lring$-theory (respectively, the $\Lval$-theory corresponding to $(K,v)$) 
has NIP.
If $K$ is a NIP field and $v$ is any valuation on $K$, then $(K,v)$
is not necessarily a NIP valued field. For example, the field $\mathbb{Q}_p$ is NIP
in $\Lring$ (in fact, $(\mathbb{Q}_p,v_p)$ is NIP, cf.~\cite[Corollaire 7.5]{Bel99}).
However, if $v$ is any prolongation of the $l$-adic valuation
on $\mathbb{Q}$ to $\mathbb{Q}_p$ (for $l\neq p$), then $(\mathbb{Q}_p,v)$
has IP (this is a special case of \cite[Theorem 5.3]{HHJ} as $v_p$ is $\Lring$-definable on $\mathbb{Q}_p$).
For pure fields, the only known algebraic consequence of being NIP is
a theorem of Kaplan, Scanlon and Wagner (\cite[Corollary 4.4]{KSW}): a NIP field
of characteristic $p>0$ admits no Galois extensions of degree
divisible by $p$. 
In this section, henselianity does not play a role.
Background on henselian NIP valued fields can be found at the beginning of the next section.

We now prove the `left-to-right' implication of our main
result (Theorem \ref{thm:1}), namely Theorem \ref{thm:2}. For this direction,
it is not necessary to assume henselianity. The main ingredient for this 
theorem, often implicit in our arguments, is the aforementioned theorem
by Kaplan, Scanlon and Wagner. From this, the equicharacteristic case of Theorem
\ref{thm:2} is straightforward:

\begin{proposition}\label{prop:equi}
If $(K,v)$ is NIP and of equal characteristic,
then $(K,v)$ is trivial or separably defectless Kaplansky.
\end{proposition}
\begin{proof}
In the case of equal characteristic zero, there is nothing to show.
So we suppose that $v$ is non-trivial and that $\mathrm{char}(K)=p>0$.
Then $(K,v)$ is Kaplansky by \cite[Proposition 4.1]{JS19}.
Let $(K^{h},v^{h})$ be the henselization of $(K,v)$.
By
\cite[Corollary 4.4]{KSW},
$K$ has no separable algebraic extensions of degree divisible by $p$.
Thus $K^{h}$ also has no separable algebraic extensions of degree divisible by $p$, since $K^{h}/K$ is separably algebraic.
By the fundamental equality (\cite[Theorem 3.3.3]{EP}),
$(K^{h},v^{h})$ is separably defectless.
Since a valued field is separably defectless if and only if its henselization is, by \cite[Theorem 18.2]{Endler},
$(K,v)$ is separably defectless.
\end{proof}

We now turn to the case of mixed characteristic.
We
quote the next two statements for the convenience of the reader.
\begin{lemma}\label{lem:Shelah}
Let $(K,v)$ be a NIP valued field, possibly with additional structure,
and let $v^0$ be a coarsening of $v$.
Then $(K,v^0,v)$ is NIP. Consequently, both $(K,v^0)$ and
$(Kv^0,\bar{v})$ are NIP.
\end{lemma}
\begin{proof}
This is a straightforward application of Shelah's expansion theorem, \cite[Corollary 3.14]{Simon}.
For more details see e.g.~\cite[Example 2.2]{Jah19}. The final claim
follows since both $(K,v^0)$ and $(Kv^0,\bar{v})$ are interpretable in 
$(K, v^0, v)$.
\end{proof}

In case the residue field of the coarsening is stably embedded as a pure field,
the converse to Lemma \ref{lem:Shelah} also holds:
\begin{proposition}[{\cite[Proposition 2.4]{JS19}}] \label{prop:JS}
Let $(K,v)$ be a valued field and $v^0$ a coarsening of $v$. Assume that
both $(K,v^0)$ and $(Kv^0, \bar{v})$ are NIP.
If the residue field
$Kv^0$ is stably embedded as a pure field in $(K,v^0)$, then $(K,v)$
is NIP.
\end{proposition}

The last ingredient needed for Theorem \ref{thm:2} is the fact that a
NIP valued field has at most one coarsening with imperfect residue field.
\begin{lemma}\label{lem:imperfect}
Let $(K,v)$ be a NIP valued field.
Then $v$ has at most one coarsening with imperfect residue field.
If such a coarsening exists, then it is the coarsest coarsening $w$ of $v$ with $\mathrm{char}(Kw)>0$.
\end{lemma}
\begin{proof}
Assume $(K,v)$ is NIP and $\mathrm{char}(Kv)=p>0$.
Let $v_{p}$ be the coarsest coarsening of $v$ with $\mathrm{char}(Kv_{p})=p$ (note that $v_{p}$ might be the trivial valuation).
Let $u$ be any coarsening of $v$.
We claim that if $u\neq v_{p}$, then $Ku$ is perfect.
Since the valuation rings of the coarsenings of $v$ are linearly ordered by inclusion,
we have either $\mathcal{O}_{u}\subsetneq\mathcal{O}_{v_{p}}$, or vice versa.
If $\mathcal{O}_{v_{p}}\subsetneq\mathcal{O}_{u}$, $\mathrm{char}(Ku)=0$, hence $Ku$ is perfect.
If $\mathcal{O}_{u}\subsetneq\mathcal{O}_{v_{p}}$, Lemma \ref{lem:Shelah} implies that $(K,v_{p},u,v)$ is NIP, and hence so is $(Kv_{p},\bar{u})$.
Since $\bar{u}$ is non-trivial by assumption, $(Kv_{p},\bar{u})$ is separably defectless Kaplansky, by Proposition \ref{prop:equi}.
In particular, $Ku$ is perfect.
\end{proof}

We are now in a position to prove that all NIP valued fields satisfy the
properties {\bf (1)} and {\bf (2)} occurring the main theorem 
(Theorem \ref{thm:1}):
\begin{theorem}\label{thm:2}
Let $(K,v)$ be a NIP valued field.
Then both of the following hold:

\MAINSTATEMENT
\end{theorem}
\begin{proof}
Assume that $(K,v)$ is NIP.
Then $Kv$ is NIP, since it is interpretable in $(K,v)$,
so {\bf(1)} holds.
We now show that one of the cases {\maincasea}, {\maincaseb}, or {\maincasec} holds.
If $\mathrm{char}(K)=\mathrm{char}(Kv)$,
then {\maincasea} is satisfied by Proposition \ref{prop:equi}.
Assume that $\mathrm{char}(K)=0$ and $\mathrm{char}(Kv)=p>0$.
We again use the standard decomposition for $(K,v)$:
\Kstandarddecomposition

By Lemma \ref{lem:Shelah},
$(Kv_{p},\bar{v})$ is NIP.
Since $(Kv_{p},\bar{v})$ is of equal characteristic $p$, it is either trivially valued or separably defectless Kaplansky, by Proposition \ref{prop:equi}.
In particular, $\rmDelta_{p}$ is $p$-divisible.

We work with the case distinction between whether or not $\rmDelta_{0}/\rmDelta_{p}$ is discrete.
If $\rmDelta_{0}/\rmDelta_{p}\cong\mathbb{Z}$, then $(K,v_{p})$ is finitely ramified, so {\maincaseb} holds.
Otherwise, $\rmDelta_{0}/\rmDelta_{p}$ is not discrete.
We let $(K^{*},v^{*})$ be an $\aleph_{1}$-saturated elementary extension of $(K,v)$,
and consider the standard decomposition
for $(K^{*},v^{*})$:
\Kstarstandarddecomposition
By Lemma \ref{lem:discrete},
$\rmDelta_{0}^{*}/\rmDelta_{p}^{*}$ is also not discrete.
By \cite[\S 4]{AK},
$\rmDelta_{0}^{*}/\rmDelta_{p}^{*}$ is isomorphic to $\mathbb{R}$.
The argument above to show the $p$-divisibility of $\rmDelta_{p}$ also applies to $\rmDelta_{p}^{*}$.
Combining these statements, $\rmDelta_{0}^{*}$ is $p$-divisible, 
which means that $(K^{*},v^{*})$ is roughly $p$-divisible, i.e., 
that any element $\gamma \in [0,v^*(p)] \subseteq v^*K^*$ is $p$-divisible.
Since this is an elementary property, $(K,v)$ is also roughly $p$-divisible,
so $\rmDelta_0$ is $p$-divisible.
To conclude that $(K,v)$ is in case {\maincasec}, it remains to show that 
$(Kv_{0},\bar{v})$ is defectless.
We have already seen that $(Kv_{p},\bar{v})$ is separably defectless, which also applies to $(K^{*}v_{p}^{*},\bar{v}^{*})$.
Next we claim that $Kv_{p}$ is perfect.
To see this: we first pass to an $\aleph_{1}$-saturated elementary extension $(K',u')\succeq(K,v_{p})$.
Since $(K,v_{p})$ is not finitely ramified, by saturation $u'$ admits a proper coarsening $w'$ with $\mathrm{char}(K'w')>0$.
Once more applying Lemma \ref{lem:Shelah}, $(K,v_{p})$ is NIP, and hence so is $(K',u')$.
By Lemma \ref{lem:imperfect}, $K'u'$ is perfect, thus $Kv_{p}$ is perfect.
Since this applies also to $K^{*}v_{p}^{*}$,
it follows that $(K^{*}v_{p}^{*},\bar{v}^{*})$ is defectless.
By \cite[\S 4]{AK},
$(K^{*}v^{*}_{0},\bar{v}_{p}^{*})$ is maximal (i.e., admits no immediate
extensions), so in particular is henselian and defectless.
Thus, by Lemma \ref{lem:defectless},
$(K^{*}v_{0}^{*},\bar{v}^{*})$ is defectless.
Applying Lemma \ref{lem:defectless} once again,
we conclude that
$(K^{*},{v}^{*})$ is defectless.
Therefore,
by Lemma \ref{lem:defect_ultrapower},
$(K,v)$ is defectless,
and so $(Kv_{0},\bar{v})$ is defectless by Lemma \ref{lem:defectless}.
This verifies that {\maincasec} holds.
\end{proof}

\begin{remark} \label{rem:3.6}
Let $(K,v)$ be a valued field and suppose that {\bf(2)} from Theorem~\ref{thm:2} holds.
If $Kv$ is finite then $(K,v)$ is trivial or finitely ramified.
To see this,
we argue as follows.
In case
{\bf(a)},
if $v$ is non-trivial
then
$(K,v)$ is separably defectless Kaplansky.
Residue fields of such valuations are in particular closed under Artin--Schreier extensions, which finite fields are not.
So $v$ is trivial.
In case
{\bf(b)},
if
$(Kv_{p},\bar{v})$ is non-trivial,
then again $Kv$ is closed under Artin--Schreier extensions,
which it is not,
therefore $\bar{v}$ is trivial.
Thus $v_{p}=v$,
and so $(K,v)$ is finitely ramified.
Finally,
if $Kv$ is finite, case
{\bf(c)} does not occur,
as finite fields are never the residue fields of Kaplansky valuations.
\end{remark}

\section{NIP transfer from residue field to valued field}
Ax-Kochen/Ershov principles, as discussed in section 2, 
allow the transfer of properties like decidability from the theories of the residue field and
value group to that of the valued field.
%
A key observation
is that this often means that also model-theoretic tameness properties, like NIP, transfer from residue
field and value group to the valued field.
The first such theorem, proven by Delon (\cite{Del81}), states that a henselian valued field of
\label{sec:trans}
equicharacteristic $0$ is NIP (in $\Lval$) if and only if its residue field and its
value group are NIP (in $\Lring$ and $\Loag$ respectively). 
By a result of Gurevich and Schmitt (\cite{GS}), the $\Loag$-theory of any ordered abelian group
is NIP.
Following Delon, several further such `NIP transfer theorems' 
have been proven for henselian valued fields: in particular by B\'elair for unramified henselian valued fields with perfect residue field (\cite[Corollaire 7.5]{Bel99}) and by Jahnke and Simon
for separably
algebraically maximal Kaplansky valued fields of finite degree of imperfection (\cite[Lemma 3.2]{JS19}). 
In this section, we prove two more analogues of Delon's theorem.
The first case that we consider is that of separably
algebraically maximal Kaplansky valued fields of infinite degree of imperfection.
The second NIP transfer theorem we prove is for henselian finitely ramified valued fields with imperfect residue field.

In both cases, we employ the proof method from \cite{JS19} (which in turn
builds on \cite{CH14}), i.e., we
consider the following two properties of a theory $T$ of valued fields from 
\cite{JS19}:
\begin{itemize}
\item[\SE]
The residue field and the value group are stably embedded.
\item[\IM]
If $K\models T$ and $a$ is a singleton in an elementary extension $(K^{*},v^{*})\succeq(K,v)$ such that $K(a)/K$, together with the restriction of $v^{*}$, is an immediate extension,
then $\mathrm{tp}(a/K)$ is implied by instances of NIP formulas.
\end{itemize}
By \cite[Theorem 2.3]{JS19}, if the theory of $(K,v)$ satisfies both 
\SE~and \IM, and $Kv$ is NIP, then $(K,v)$ is NIP.

\subsection*{NIP transfer for separably algebraically maximal Kaplansky valued fields}

\begin{proposition}\label{prp:SAMK}
Let $(K,v)$ be a valued field of equal characteristic $p>0$ which is
separably algebraically maximal and Kaplansky.
Then
$(K,v)$ is NIP
if and only if $Kv$ is NIP.
\end{proposition}
\begin{proof}
Clearly, if $(K,v)$ is NIP then $Kv$ is NIP.
We now show that the theory of any separably maximal Kaplansky field of
positive characteristic satisfies both properties \SE~and \IM.

If $K$ has finite degree of imperfection, $(K,v)$ satisfies \SE~by 
\cite[Lemma 3.1]{JS19}.
For the case of infinite degree of imperfection, note that
by \cite[Th\'{e}or\`{e}me 3.1]{Del82},
the theory of separably algebraically maximal Kaplansky valued fields of characteristic $p$ and given imperfection degree (allowed to be infinite) with value group elementarily equivalent to $vK$ and residue field elementarily equivalent to $Kv$ is complete.
Exactly as explained in the proof of \cite[Lemma 3.1]{JS19},
the theory of $(K,v)$ satisfies \SE~also in case the degree of imperfection 
is infinite; only citing Delon
(\cite[Th\'{e}or\`{e}me 3.1]{Del82})
rather than Kuhlmann and Pal 
(\cite[Theorem 5.1]{KuhlmannPal}).

In \cite[Lemma 3.2]{JS19}, \IM~is proved in the case that the imperfection degree of $K$ is finite.
Accordingly, we suppose that $K$ has infinite degree of imperfection.
Let $K(a)/K$ be an immediate extension, taken within an elementary extension $(K^{*},v^{*})\succeq(K,v)$,
and from now on we equip all subfields of $K^{*}$ with the restriction of $v^{*}$.
Since $K(a)/K$ is immediate, $K(a)$ is also Kaplansky.
We will show that the type $\mathrm{tp}(a/K)$ is implied by quantifier-free formulas.
If $a$ is algebraic over $K$ then already $a$ is a member of $K$, and the rest of the argument is trivial.
Otherwise, we suppose that $a$ is transcendental over $K$.
Let $M$ be a maximal algebraic purely wild extension of $K(a)$, taken within $K^{*}$.
Since $K(a)$ is Kaplansky, $M$ is uniquely determined up to isomorphism over $K(a)$,
by \cite[Theorem 5.1]{KPR}.
Now let $L/K(a)$ denote the subextension of $M/K(a)$ consisting of those elements of $M$ that are separably algebraic over $K(a)$.
Then $L/K$ is again immediate, and $L$ is a maximal separably algebraic purely wild extension of $K(a)$, which determines $L$ uniquely up to isomorphism over $K(a)$.
Moreover, $L$ is separably algebraically maximal and Kaplansky.
Since $L/K$ is separable, the imperfection degree of $L$ is also infinite.
By \cite[Th\'{e}or\`{e}me 3.1]{Del82},
we have $K\preceq L$ as valued fields.
This shows that the quantifier-free type of $a$ over $K$ determines the isomorphism type of an elementary submodel of $K^{*}$ which contains $a$, thus the quantifier-free type of $a$ over $K$ implies the full type of $a$ over $K$.
Recall that any quantifier-free $\Lval$-formula has NIP since 
every non-trivially valued field embeds into an 
algebraically closed non-trivially valued field and the $\Lval$-theory ACVF has NIP,
see \cite[Theorem A.11]{Simon}. Thus,
$(K,v)$ has the property \IM.
Therefore, if $Kv$ is NIP in the language of rings, $(K,v)$ is NIP in the
language of valued fields, by {\cite[Theorem 2.3]{JS19}}.
\end{proof}

As a special case of the previous Proposition, we get that the theory of any
separably closed valued field is NIP. This was previously shown by Delon 
(although her proof remains unpublished), and -- in the case of finite
degree of imperfection -- by Hong (\cite[Corollary 5.2.13]{Hon13}).
\begin{corollary}
The complete theory of any separably closed valued field is NIP. \label{cor:SCVF}
\end{corollary}
\begin{proof}
Let $(K,v)$ be any separably closed valued field. Then $(K,v)$ is separably 
algebraically maximal and Kaplansky, since $K$ has no non-trivial 
separable field extensions. As $Kv$ is separably closed, it is
stable by \cite{Wood}. 
In particular, $Kv$ is NIP and so $(K,v)$ is NIP by Proposition
\ref{prp:SAMK}.
\end{proof}

\subsection*{NIP transfer for henselian finitely ramified valued fields}
Recall that we call a valued field $(K,v)$ of mixed characteristic $(0,p)$ unramified if $v(p)$ is the minimum positive element of the value group, and finitely
ramified if the interval $(0,v(p)] \subseteq vK$ is finite.
In particular, any henselian valued field of mixed characteristic with
value group $\mathbb{Z}$ is finitely ramified, and so is any power series field
over such a field (valued by the composition of the power series valuations and the
mixed characteristic valuation with value group $\mathbb{Z}$).\footnote{In fact, every henselian finitely ramified valued field is elementarily
equivalent to such a composition, as we argue in the proof of Proposition \ref{prp:elem}.}

For unramified henselian valued fields $(K,v)$ with perfect residue field
$Kv$, B\'elair proved that the valued field is NIP in case its residue
field is NIP in \cite[Th\'eor\`eme 7.4(2)]{Bel99}. Although he does not state
the assumption explicitly that the residue field need be perfect, his proof
relies crucially on properties of Witt rings which only hold over perfect fields.
We use Cohen rings (see section 2) instead.

We first show a NIP transfer result for henselian unramified valued fields
(allowing the residue field to be imperfect). In order to do this, 
we verify once more that the conditions \SE~and \IM~hold.
The fact that the residue field and value group are stably embedded as a pure field  or, respectively, a pure ordered abelian group, in
a henselian unramified valued field was shown in 
\cite[Theorem 1.4]{AJ2},
so we only have to 
show \IM.


\begin{lemma}\label{lem:IM}
Let $(K,v)$ be a henselian unramified valued field.
Then the $\Lval$-theory of $(K,v)$ satisfies \IM.
\end{lemma}
\begin{proof}
Let $(K,v)$ be a henselian unramified valued field,
and let $(L,w)$ be elementarily equivalent to $(K,v)$.
Take any $a\in(L^{*},w^{*})\succeq(L,w)$ such that $(M,\omega):=(L(a),w^{*}|_{L(a)})$ is an immediate extension of $(L,w)$.
We show that the type of $a$ over $L$ is determined by the quantifier-free type of $a$ over $L$.
The isomorphism type of $a$ over $L$ determines the henselization $(M^{h},\omega^{h})$ of $(M,\omega)$ up to isomorphism over $(L,w)$.
Since the extensions of the value group and residue field are trivial, they are elementary.
Thus,
by \cite[Theorem 4.3.4]{ErshovMult}, $(L,w)\preceq(M^{h},\omega^{h})$.
This shows that the quantifier-free type of $a$ over $L$ completely determines a model containing $L(a)$, and thus determines the complete type of $a$ over $L$.
\end{proof}

Applying the two lemmas above, we now get a NIP transfer principle for henselian unramified valued fields.

\begin{proposition}\label{prp:unramified_transfer}
Let $k$ be a NIP field of characteristic $p>0$.
Then any henselian unramified valued field $(K,v)$, with residue field $k$, is NIP
in the language $\Lval$ of valued fields.
\end{proposition}
\begin{proof}
Let $(K,v)$ be a henselian unramified valued field with NIP residue field.
We verify that the conditions \SE~and \IM~hold for the theory of $(K,v)$.
First, \SE~holds by \cite[Theorem 1.4]{AJ2}.
By Lemma \ref{lem:IM}, property \IM~also holds.
Hence, applying \cite[Theorem 2.3]{JS19}, $(K,v)$ is NIP.
\end{proof}

We deduce a version of NIP transfer for mixed characteristic valued fields $(K,v)$ in the case that $(K,v_{p})$ is finitely ramified: Here, $\SE$ cannot be deduced by showing that 
any automorphism of the residue field (resp.~value group) lifts to an automorphism 
of the valued field, as in finitely ramified fields the existence of such a lift may fail 
(see \cite[Example 11.5]{AJ2}).\footnote{That $\SE$ holds for finitely
ramified fields has since been shown in \cite[Theorem 6.2 and Remark 6.3]{FPS}.}
Thus, we take a different route.

Before we start, we need some further details about finitely ramified fields.
The key idea is that - up to elementary equivalence - every henselian finitely ramified
valued field with value group elementarily equivalent to $\mathbb{Z}$ is in fact a finite extension of a henselian unramified valued field over the same residue field.
This will be made precise in Lemma \ref{lem:finite_ramification}.

\begin{fact}[{\cite[Theorem 22.7]{War93}}]\label{fact:Warner}
Let $(K,v)$ be a complete valued field of mixed characteristic with
value group $vK\cong\mathbb{Z}$ and ramification $e>0$, that is, the interval $(0,v(p)] \subseteq vK$ contains $e$ many elements.
Then $(K,v)$ is an extension of degree $e$ of a complete unramified valued field $(C(Kv),w)$
which has residue field $Kv$ and value group $\mathbb{Z}$.
The latter is a called Cohen subfield of $(K,v)$.
\end{fact}

We can now use Cohen subfields to tackle NIP transfer in finitely ramified
henselian valued fields:
\begin{proposition}\label{prp:fin_ram}
Let $(K,v)$ be a henselian valued field of mixed characteristic $(0,p)$
such that $(K,v_{p})$ is finitely ramified
and $(Kv_{p},\bar{v})$ is NIP.
Then $(K,v)$ is NIP.
\end{proposition}
\begin{proof}
Let $(K^{*},v^{*})\succeq(K,v)$ be an $\aleph_{1}$-saturated elementary extension.
Since $v_{p}$ is $\Lring$-definable in $K$, the corresponding valuation $v_{p}^{*}$ on $K^{*}$ is also finitely ramified, and $(K^{*}v_{p}^{*},\bar{v}^{*})$ is NIP.
Also if $(K^{*},v^{*})$ is NIP then $(K,v)$ is NIP.  
Therefore, without loss of generality, we may suppose from now on that $(K,v)$ is $\aleph_{1}$-saturated.

Consider the standard decomposition of $(K,v)$:
\Kstandarddecomposition
The rank-$1$ valued field $(Kv_{0},\bar{v}_{p})$ is complete and finitely ramified.
By Fact \ref{fact:Warner},
there is a subfield $L$ of $Kv_{0}$ with $Kv_{0}/L$ finite, such that the restriction $w$ of $\bar{v}_{p}$ to $L$ is unramified and complete, and $Lw=Kv_{p}$.
Hence $(L,w)$ is NIP, by Proposition \ref{prp:unramified_transfer}.
Moreover, the residue field $Lw$ is stably embedded in $(L,w)$, by \cite[Theorem 1.4]{AJ2}.
Thus by Proposition \ref{prop:JS}, $(L,w')$ is NIP, where
the composition $w':=\bar{v}\circ w$ is a valuation on $L$ with residue field $Lw'=Kv$.
In fact $(Kv_{0},\bar{v})/(L,w')$ is a finite extension of valued fields,
and the valuation ring of $\bar{v}$ is the integral closure in $Kv_{0}$ of the valuation ring of $w'$, since $w'$ is henselian.
Thus $(Kv_{0},\bar{v})$ is interpretable in $(L,w')$,
and it follows that $(Kv_{0},\bar{v})$ is NIP.
Furthermore, $(K,v_{0})$ is NIP, since it is an equicharacteristic zero henselian valued field with NIP residue field,
\cite[Theorem A.15]{Simon}.
Moreover, the residue field $Kv_{0}$ is stably embedded in $(K,v_{0})$, by \cite[Corollary 5.25]{vdD}.
Applying Proposition \ref{prop:JS} again,
we conclude that $(K,v)$ is NIP.
\end{proof}

As a special case, we immediately get the following corollary.

\begin{corollary}\label{cor:fin_ram_2}
Let $(K,v)$ be a henselian finitely ramified valued field of mixed characteristic with $Kv$ NIP.
Then $(K,v)$ is NIP in the language $\Lval$ of valued fields.
\end{corollary}

\section{The main theorem and immediate consequences}

We are now in a position to prove our main theorem. \label{sec:mt}

\begin{theorem}\label{thm:1}
Let $(K,v)$ be a henselian valued field.
Then $(K,v)$ is NIP if and only if both of the following hold:

\MAINSTATEMENT
\end{theorem}
\begin{proof}
The implication $\implies$ is a special case of Theorem \ref{thm:2}.
It remains to prove the converse,
thus from now on we assume that $Kv$ is NIP
and $(K,v)$ satisfies one of
{\maincasea},
{\maincaseb},
or
{\maincasec}.

In case {\maincasea},
if $v$ is trivial, we automatically get that $(K,v)$ is NIP.
Otherwise, $v$ is separably defectless Kaplansky and henselian.
Thus it is separably algebraically maximal and Kaplansky.
Now Proposition~\ref{prp:SAMK} implies that $(K,v)$ is NIP.

If $(K,v)$ satisfies {\maincaseb},
then $(Kv_{p},\bar{v})$ is either trivial (in particular $Kv_p=Kv$) and thus NIP,
or separably defectless Kaplansky.
In the latter case, $(Kv_p,\bar{v})$ is NIP by Proposition \ref{prp:SAMK}.
Applying Proposition \ref{prp:fin_ram}, we conclude that $(K,v)$ is NIP.

Finally, assume $(K,v)$ is in case {\maincasec}.
We show first that $(Kv_{0},\bar{v})$ is NIP.
Either $\bar{v}$ is trivial, or $(Kv_{0},\bar{v})$ is defectless Kaplansky and henselian, thus algebraically maximal.
In either case, $(Kv_{0},\bar{v})$ is NIP (\cite[Theorem 3.3]{JS19}). In particular,
$Kv_0$ is NIP.
Once more, $(K,v_{0})$ is NIP by Delon's Theorem (\cite[Theorem A.15]{Simon}). Furthermore, 
$Kv_{0}$ is stably embedded as a pure field, see \cite[Corollary 5.25]{vdD}.
Finally, applying Proposition \ref{prop:JS} once again,
we conclude that $(K,v)$ is NIP.
\end{proof}

Note that none of the cases appearing in the theorem
is vacuous:

\begin{example} \label{Ex:NIP}
\begin{enumerate}
\item[{\maincasea}]
    We give examples of NIP valued fields $(K,v)$ of equal characteristic where the valuation is trivial or separably defectless Kaplansky.
    Naturally the fields $\mathbb{C}$, $\mathbb{R}$, and $\mathbb{F}^{\mathrm{alg}}_{p}$,
    equipped with the trivial valuation,
    are suitable examples,
    as well as $\mathbb{C}((\rmGamma))$ and $\mathbb{R}((\rmGamma))$ with the power series valuation $v_{\rmGamma}$, for any ordered abelian group $\rmGamma$.
    Moreover, if $\rmGamma$ is $p$-divisible then $\mathbb{F}_{p}^{\mathrm{alg}}((\rmGamma))$ together with the power series valuation $v_{\rmGamma}$ is NIP.
    In particular, $\mathrm{ACVF}_{0,0}$, $\mathrm{SCVF}_{p}$, and $\mathrm{RCVF}$ are NIP.
\item[{\maincaseb}]
    We give examples of NIP valued fields $(K,v)$ of mixed characteristic $(0,p)$ such that $(K,v_{p})$ is finitely ramified and $(Kv_{p},\bar{v})$ is trivial or separably defectless Kaplansky.
    The most basic examples are $\mathbb{Q}_{p}$ with the $p$-adic valuation $v_{p}$, and any finite extension thereof.
    In all of these examples, $(Kv_{p},\bar{v})$ is trivial.
    To illustrate the case where $(Kv_{p},\bar{v})$ is separably defectless Kaplansky and non-trivial,
    we start with any separably closed non-trivially valued field $(k,u)$ of characteristic $p$.
    In particular, $(k,u)$ is NIP by  Proposition \ref{prp:SAMK}.
    Now, let $(K,w)$ be a Cohen field over $k$
    and consider the valuation $v$ on $K$ defined to be the composition $u\circ w$.
    Note that $v_{p}=w$ since $w$ is the coarsest mixed characteristic coarsening of $v$.
    Since $(k,u)=(Kv_p, \bar{v})$ is NIP, $(K,v)$ is NIP by Proposition \ref{prp:fin_ram}.
    If $k$ is imperfect, then $(Kv_p,\bar{v})$ is separably algebraically maximal but not algebraically
    maximal.
    Furthermore, any (generalized) power series field over any of the aforementioned examples, together with the composition of valuations,
    is again NIP.
\item[{\maincasec}]
    We give examples of NIP valued fields $(K,v)$ of mixed characteristic $(0,p)$ such that $(Kv_{0},\bar{v})$ is defectless Kaplansky.
    The most obvious examples are algebraically closed valued fields, i.e.~models of $\mathrm{ACVF}_{0,p}$.
    More generally, given any perfect infinite NIP field $k$ of characteristic $p$, e.g.~$\mathbb{F}_{p}^{\mathrm{alg}}$, we may construct suitable $(K,v)$ with residue field $k$, as follows.
    
    Let $\rmGamma$ be a non-trivial $p$-divisible ordered abelian group.
    By a standard construction, see for example 
    \cite[Theorem 2.14]{Kuh04},
    there is a valued field $(L,w)$ of mixed characteristic $(0,p)$ with value group $\rmGamma$ and residue field $k$.
    Now let $(K,v)$ be a maximal algebraic purely wild extension of $(L,w)$.
    Since $\rmGamma$ is $p$-divisible and $k$ admits no finite extensions of degree divisible by $p$,
    $(K,v)$ is the unique such extension up to isomorphism over $L$,
    by \cite[Theorem 5.1]{KPR}.
    Moreover, 
    $(K,v)/(L,w)$ is immediate, by \cite[Lemma 5.2]{KPR}.
    By construction, $(K,v)$ is defectless Kaplansky, and it is NIP by \cite[Theorem 3.3]{JS19}.
    As before, any (generalized) power series field over an example as just described, together with the composition of valuations, is NIP.
\end{enumerate}
\end{example}

As a consequence of Theorem \ref{thm:1}, we obtain an analogue to a result by Halevi and Hasson,
who proved that the henselization of every strongly dependent
valued field is strongly dependent (\cite[Theorem 2]{HH2}):

\begin{corollary}\label{cor:3}
If $(K,v)$ is NIP, then its henselization $(K^{h},v^{h})$ is NIP.
\end{corollary}
\begin{proof}
Let $(\dagger)$ be one of the following properties of a valued field which all occur in the statement
of Theorem \ref{thm:1}:
\begin{enumerate}
\item trivially valued
\item of equal characteristic zero,
\item of equal characteristic $p$, for a given prime $p>0$,
\item of mixed characteristic $(0,p)$, for a given prime $p>0$,
\item separably defectless
\item defectless
\item Kaplansky
\item finitely ramified (for $(K,v)$ of mixed characteristic).
\end{enumerate}
\begin{claim}\label{claim:1}
If $(K,v)$ satisfies $(\dagger)$ then $(K^{h},v^{h})$ satisfies $(\dagger)$.
\end{claim}
\begin{claimproof}[Proof of claim]
If $v$ is trivial then it is already henselian.
The properties relating to characteristic, i.e.~{\bf(ii)--(iv)}, are preserved when taking any extension, so in particular when passing to the henselization.
For `separably defectless' we apply \cite[Theorem 18.2]{Endler}.
For `defectless' we argue as follows:

First we consider a finite purely inseparable extension $L/K$,
and denote by $w$ the unique prolongation of $v$ to $L$.
The compositum $LK^{h}$ -- equipped with the unique prolongation of $v^{h}$ -- coincides with the henselization $(L^{h},w^{h})$ of $(L,w)$.
Since $K^{h}/K$ is separably algebraic, $L/K$ is linearly disjoint from $K^{h}/K$, and so $[L^{h}:K^{h}]=[L:K]$.
Moreover any finite purely inseparable extension of $K^{h}$ arises in this way.
Since both $(L^{h},w^{h})/(L,w)$ and $(K^{h},v^{h})/(K,v)$ are immediate,
we have the equivalence
\begin{align*}
	[L:K] = e(w/v)\,f(w/v) &\Longleftrightarrow [L^{h}:K^{h}] = e(w^{h}/v^{h})\,f(w^{h}/v^{h}).
\end{align*}
Combining this with the claim for the property of separable defectlessness, it follows that if $(K,v)$ is defectless then $(K^{h},v^{h})$ is defectless.

Finally, the properties `Kaplansky' and `finitely ramified' each depend only on the value group and residue field.
Since $(K^{h},v^{h})/(K,v)$ is an immediate extension, it follows that if $(K,v)$ is Kaplansky (respectively, finitely ramified) then $(K^{h},v^{h})$ satisfies the same property.
\end{claimproof}
Let $(K,v)$ be NIP.
By Theorem \ref{thm:2}, $(K,v)$ is in one of the cases {\maincasea}, {\maincaseb}, or {\maincasec}.
By repeated application of the claim, we will now show that $(K^{h},v^{h})$ is in the same case as $(K,v)$.
First, we consider case {\maincasea}.
There are four properties of $(K,v)$ mentioned in {\maincasea}, and each of those properties is shown in the claim to transfer from $(K,v)$ to $(K^{h},v^{h})$.
Thus if $(K,v)$ satisfies {\maincasea}, then $(K^{h},v^{h})$ also satisfies {\maincasea}.

Also by the claim, if $(K,v)$ is of mixed characteristic $(0,p)$,
then the same is true of $(K^{h},v^{h})$.
We now consider the standard decomposition of $(K^{h},v^{h})$:
\Khstandarddecomposition
Note that $vK=v^{h}K^{h}$ and $Kv=K^{h}v^{h}$,
from which we have $\rmDelta_{0}=\rmDelta_{0}^{h}$ and $\rmDelta_{p}=\rmDelta_{p}^{h}$.
Nonetheless, $v_{0}^{h}$ does not denote the henselization of $v_{0}$, but the coarsening of $v^{h}$ corresponding to the convex subgroup $\rmDelta_{0}^{h}$;
and likewise for $v_{p}^{h}$.
The coarsening $(K,v_{p})$ is finitely ramified if and only if $\rmDelta_{0}/\rmDelta_{p}$ is discrete,
which is purely a property of the value group $vK$.
Since $\rmDelta_{0}/\rmDelta_{p}=\rmDelta_{0}^{h}/\rmDelta_{p}^{h}$,
we have that
$(K,v_{p})$ is finitely ramified if and only if $(K^{h},v^{h}_{p})$ is also finitely ramified.
By Lemma \ref{lem:henselization}{\bf(ii)}, the henselization of $(Kv_{p},\bar{v})$ is $(K^{h}v^{h}_{p},\bar{v}^{h})$.
By the claim, if $(Kv_{p},\bar{v})$ is trivial or separably defectless Kaplansky, so is $(K^{h}v^{h}_{p},\bar{v}^{h})$.
Thus if $(K,v)$ satisfies {\maincaseb}, then $(K^{h},v^{h})$ also satisfies {\maincaseb}.

On the other hand, by Lemma \ref{lem:henselization}{\bf(ii)}, the henselization of $(Kv_{0},\bar{v})$ is $(K^{h}v_{0}^{h},\bar{v}^{h})$.
By the claim, if $(Kv_{0},\bar{v})$ is defectless Kaplansky, so is $(K^{h}v_{0}^{h},\bar{v}^{h})$.
Thus if $(K,v)$ satisfies {\maincasec}, then $(K^{h},v^{h})$ also satisfies {\maincasec}.

We have shown that each of the conditions 
{\maincasea}, {\maincaseb}, and {\maincasec} is preserved by taking the henselization.
Moreover, $Kv$ is equal to $K^{h}v^{h}$, so
one is NIP if and only if so is the other.
Therefore, since $(K,v)$ satisfies the conjunction of {\bf(1)} and {\bf(2)}, so does $(K^{h},v^{h})$.
Since $(K^{h},v^{h})$ is henselian,
it follows
from Theorem \ref{thm:1}
that $(K^{h},v^{h})$ is NIP.
\end{proof}

\section{A model-theoretic version of Theorem \ref{thm:1}}

The aim for this section is to resolve the following: \label{sec:mtv}
\begin{task} Given a complete $\Lring$-theory $T_{k}=\mathrm{Th}(k)$ of NIP fields, describe all of the complete $\Lval$-theories of NIP henselian valued fields $(K,v)$ such that the residue field $Kv$ is 
a model of $T_{k}$.
\end{task}

\subsection*{Equal characteristic} Fix a complete $\Lring$-theory $T_k$ of NIP
fields.
Of course, one complete theory of NIP henselian valued fields with residue field a model of $T_k$
is the theory of $k$ equipped with the trivial valuation.
If $\mathrm{char}(k)=0$, then -- by the Ax--Kochen/Ershov Theorem in equicharacteristic $0$ (\cite[Theorem 5.11]{vdD}) -- for each complete $\Loag$-theory $T_{\rmGamma}=\mathrm{Th}(\rmGamma)$ of non-trivial ordered abelian groups,
there is exactly one complete $\Lval$-theory
of equicharacteristic zero henselian valued fields $(K,v)$
with residue field $Kv\models T_{k}$ and value group $vK\models T_{\rmGamma}$,
namely the theory of $(k((\rmGamma)),v_{\rmGamma})$.
We denote this theory by $T(k,\rmGamma)$.
Vacuously, models of $T(k,\rmGamma)$ are separably defectless and Kaplansky.

If $k$ is of positive characteristic $p$, is perfect and admits no Galois extensions of degree divisible by $p$,
then for each complete $\Loag$-theory $T_{\rmGamma}=\mathrm{Th}(\rmGamma)$ of $p$-divisible non-trivial ordered abelian groups,
and each $e\in\mathbb{N}\cup\{\infty\}$,
by \cite[Th\'{e}or\`{e}me 3.1]{Del82}
there is exactly one complete $\Lval$-theory
of equicharacteristic $p$ henselian separably defectless valued fields
$(K,v)$ of imperfection degree $e$,
and with residue field $Kv\models T_{k}$ and value group $vK\models T_{\rmGamma}$,
which we denote by $T^{\mathrm{sd}}_{e}(k,\rmGamma)$.
Models of $T^{\mathrm{sd}}_{e}(k,\rmGamma)$ will be Kaplansky, by our assumptions on $k$ and $\rmGamma$.
Note that if $k$ is imperfect, it is not the residue field of a Kaplansky valued field, and
hence we necessarily have that $K=k$ and $v$ is the trivial valuation.

By Theorem \ref{thm:1}, these are all the complete $\Lval$-theories of
NIP valued fields in case {\maincasea}. Thus, we have determined all complete 
$\Lval$-theories of NIP henselian valued fields of equal characteristic
with residue field a model of $T_k$.

\subsection*{Mixed characteristic}
A complete theory of henselian valued fields of 
mixed characteristic $(0,p)$ does not only depend on the 
complete theory of the value 
group and the residue field but also on the value of $p$. 
Just like before, we use the standard decomposition to differentiate between
the cases {\maincaseb} and {\maincasec}. As above, we fix a 
complete $\Lring$-theory $T_{k}=\mathrm{Th}(k)$ of NIP fields of characteristic
$p>0$. 

Now, consider a triple $(k,\rmGamma,\gamma)$ where $\rmGamma$ is an ordered abelian group, and $\gamma\in\rmGamma$ is such that $\gamma>0$. Our aim is to 
characterize the complete theories of NIP henselian valued fields 
$(K,v)$ such that
$Kv\equiv k$
and $(vK,v(p))\equiv(\rmGamma,\gamma)$.
Mimicking the standard decomposition, but expressed purely for the ordered abelian group $\rmGamma$, rather than for a valued field,
we let $\rmGamma_{\gamma+}$ be the smallest convex subgroup of $\rmGamma$ 
containing $\gamma$, and let $\rmGamma_{\gamma-}$ be the greatest convex subgroup of $\rmGamma$ not containing $\gamma$.

\subsubsection*{Complete theories at the heart of case {\maincaseb}}
Let $e\in\mathbb{N}\cup\{\infty\}$
and suppose that the image of $\gamma$ is minimum positive in $\rmGamma/\rmGamma_{\gamma-}$.
Note that this is an elementary property of the ordered abelian group $\rmGamma$ with a constant symbol for $\gamma$.
Let $T_{e}(k,\rmGamma,\gamma)$
be the theory of valued fields $(K,v)$ of mixed characteristic $(0,p)$
such that
\begin{enumerate}
\item $Kv\equiv k$,
\item $(vK,v(p))\equiv(\rmGamma,\gamma)$,
\item $(Kv_{p},\bar{v})$ is separably algebraically maximal of 
	imperfection degree $e$, and
\item $(K,v_{p})$ is henselian.
\end{enumerate}
Recall that the valuation $v_{p}$ is $\Lring$-definable in $K$, without parameters,
uniformly for all henselian unramified valued fields $(K,v_{p})$ of mixed characteristic $(0,p)$.
Thus, the above listed properties of $(K,v)$ are $\Lval$-axiomatizable:
the axiomatizability of {\bf(i)} and {\bf(ii)} simply uses $T_{k}$ and $T_{\rmGamma}$,
for {\bf(iii)} and {\bf(iv)} we use the uniform definability of $v_p$ plus (for example) the axioms discussed in \cite[Section 4]{Kuh16}.

\begin{lemma}
If $k$ is infinite, perfect and NIP,
and if $\rmGamma_{\gamma-}$ is $p$-divisible,
then
$T_{e}(k,\rmGamma,\gamma)$ is complete. 
\end{lemma}
\begin{proof}
Let $(K,v)$ be a model of $T_{e}(k,\rmGamma,\gamma)$.
Since $k$ is infinite, perfect, and NIP of positive characteristic,
it admits no finite extension of degree divisible by $p$.
Thus $(Kv_{p},\bar{v})$ is a separably algebraically maximal Kaplansky valued field of equal characteristic $p$ and imperfection degree $e$.
By \cite[Th\'{e}or\`{e}me 3.1]{Del82},
$T_{e}(k,\rmGamma,\gamma)$ determines the complete theory of $(Kv_{p},\bar{v})$.
Moreover, by \cite[Theorem 1.2]{AJ2} the complete $\Lval$-theory of $(K,v_{p})$ is determined by henselianity and the theories of the residue field and the value group, which all follow from $T_{e}(k,\rmGamma,\gamma)$.
\end{proof}



\subsubsection*{Complete theories at the heart of case {\maincasec}}
Now suppose that $\rmGamma/\rmGamma_{\gamma-}$ is not discrete.
Again, this is an elementary property of the ordered abelian group $\rmGamma$ with a constant symbol for $\gamma$, see the proof of Lemma \ref{lem:discrete}.
We denote by $[\rmGamma]_{\gamma}$ the relative divisible hull in $\rmGamma$ of the subgroup generated by $\gamma$.
Note that $[\rmGamma]_{\gamma}$ will always be a subgroup of $\rmGamma_{\gamma+}$, although it will in general not be a convex subgroup.

We say that a valued field $(F,v_{F})$ is {\em compatible} with $(\rmGamma,\gamma)$ if it is algebraically maximal, an algebraic extension of $(\mathbb{Q},v_{p})$, its residue field $Fv_{F}$ is $\widetilde{\mathbb{F}}_{p}$, and its value group satisfies $(v_{F}F,v_{F}(p))\cong([\rmGamma]_{\gamma},\gamma)$.
If $\rmGamma_{\gamma+}$ is $p$-divisible, then $[\rmGamma]_{\gamma}$ is $p$-divisible, and so in this case $(F,v_{F})$ will be Kaplansky.
Let $T_{}(k,\rmGamma,\gamma)$ be the theory of valued fields $(K,v)$ of mixed characteristic $(0,p)$ such that
\begin{enumerate}
\item $Kv\equiv k$,
\item $(vK,v(p))\equiv(\rmGamma,\gamma)$, and
\item $(K,v)$ is algebraically maximal.
\end{enumerate}
Given compatible $(F,v_{F})$, we let $T_{}(k,\rmGamma,\gamma,F,v_{F})$ be the theory extending $T_{}(k,\rmGamma,\gamma)$ by further requiring of $(K,v)$ that
\begin{enumerate}
\setcounter{enumi}{3}
\item the algebraic part of $(K,v)$ is isomorphic to $(F,v_{F})$.
\end{enumerate}
Again, these properties are obviously $\Lval$-axiomatizable (for {\bf(iii)} see \cite[Section 4]{Kuh16}).
The description of the complete theories of algebraically maximal Kaplansky valued fields of mixed characteristic is well-known, and due independently to Ershov and Ziegler.
However, for lack of a convenient reference, in the proof of the following lemma we rely on more modern sources.

\begin{lemma}\label{lem:Ershov.Ziegler}
If $k$ is infinite, perfect and NIP,
and if $\rmGamma$ is $p$-divisible,
then 
$T_{}(k,\rmGamma,\gamma,F,v_{F})$ is complete.
\end{lemma}
\begin{proof}
Let $(K,v)$ and $(L,w)$ be models of $T(k,\rmGamma,\gamma,F,v_{F})$.
Again we argue that, since $k$ is infinite, perfect, and NIP of positive characteristic,
it admits no finite extension of degree divisible by $p$.
Thus $(K,v)$ and $(L,w)$ are algebraically maximal Kaplansky valued fields of mixed characteristic $(0,p)$: in particular they are tame.
We may identify $(F,v_{F})$---which is also tame---with the algebraic part of both $(K,v)$ and $(L,w)$.
Since the class of tame fields is relatively subcomplete, by \cite[Theorem 7.3]{Kuh16}, we have
\begin{align*}
	(K,v)\equiv_{(F,v_{F})}(L,w).
\end{align*}
In particular $T(k,\rmGamma,\gamma,F,v_{F})$ is complete.
\end{proof}

Given a valued field $(K,v)$ of mixed characteristic $(0,p)$ with value group $\rmGamma$,
we denote by $\rmGamma_{(p)}$ the maximal $p$-divisible convex subgroup of $\rmGamma$, and by $v_{(p)}$ the corresponding coarsening of $v$.
Then $v$ induces a valuation $\bar{v}$ on the residue field $Kv_{(p)}$ which has value group $\rmGamma_{(p)}$,
and the value group of $v_{(p)}$ is isomorphic to $\rmGamma/\rmGamma_{(p)}$.

\begin{lemma}\label{lem:p-define}
There is an $\Lval$-formula $\pi_{p}(x)$ in the language of ordered abelian groups which defines the convex subgroup $\rmGamma_{(p)}$ in any ordered abelian group $\rmGamma$.
Thus the valuation ring corresponding to the valuation $v_{(p)}$ is uniformly $\Lval$-definable, without parameters, in all valued fields.
\end{lemma}
\begin{proof}
Take $\pi_{p}(x)$ to be the formula
$\forall y\;(0\leq y\leq|x|\longrightarrow\exists z\;pz=y)$.
\end{proof}

\begin{lemma}\label{lem:complete.composition}
Let $T$ be a theory of bivalued fields $(K,v',v)$ with $v'$ an equal characteristic zero henselian coarsening of $v$, and suppose that $T$ entails complete theories of valued fields $(K,v')$ and $(Kv',\bar{v})$.
Then $T$ is complete.
\end{lemma}
\begin{proof}
Let $(K,v',v)$ and $(L,w',w)$ be two saturated 
models of $T$ of the same cardinality.
By the saturation assumption, 
we may assume that there are
isomorphisms $\psi:(K,v')\longrightarrow(L,w')$ and 
$\phi:(Kv',\bar{v})\longrightarrow(Lw',\bar{w})$.
By stable embeddedness of the residue field in equal characteristic zero as a pure field, there is an isomorphism 
$\chi: (K,v')\longrightarrow(L,w')$
inducing $\phi$. By construction, $\chi$ is also an isomorphism $(K,v)\longrightarrow(L,w)$. Thus, $T$ is complete.
\end{proof}

The next lemma is a very simple modification of
Lemma \ref{lem:Ershov.Ziegler}.

\begin{lemma}
If $k$ is infinite, perfect and NIP,
and if $[-\gamma,\gamma]\subseteq p\rmGamma$,
then
$T_{}(k,\rmGamma,\gamma,F,v_{F})$ is complete.
\end{lemma}
\begin{proof}
Let $(K,v)$ be a model of $T_{}(k,\rmGamma,\gamma,F,v_{F})$,
and we apply the standard decomposition to $(K,v)$ with the usual notation:
\Kstandarddecomposition
The assumption $[-\gamma,\gamma]\subseteq  p\rmGamma$ entails that $\rmDelta_{0}\subseteq vK_{(p)}$,
which in turn means that $v_{(p)}$ is a coarsening of $v_{0}$.
Then $(Kv_{(p)},\bar{v})$ is of mixed characteristic $(0,p)$ and is algebraically maximal, since the composition $v_{(p)}\circ\bar{v}=v$ is algebraically maximal.
Moreover, its value group $\bar{v}(Kv_{(p)})$ is $p$-divisible and elementarily equivalent to $\rmGamma_{(p)}$, by Lemma \ref{lem:p-define};
and its residue field $(Kv_{(p)})\bar{v}=Kv$ admits no extension of degree divisible by $p$
since it is elementarily equivalent to the infinite NIP field $k$.
Thus $(Kv_{(p)},\bar{v})$ is Kaplansky.
By Lemma \ref{lem:Ershov.Ziegler}, the property `algebraically maximal', together with the theories of $k$ and $(\rmGamma_{(p)},\gamma)$,
determines the complete $\Lval$-theory of $(Kv_{(p)},\bar{v})$.
The $\Lval$-theory of the equal characteristic zero valued field
$(K,v_{(p)})$ is determined by its henselianity and by the theories of $Kv_{(p)}$ and $v_{(p)}K$, by the Ax--Kochen/Ershov Principle (\cite[Theorem 5.11]{vdD}).
From Lemma \ref{lem:complete.composition}, it follows that $T(k,\rmGamma,\gamma,F,v_{F})$ is complete.
\end{proof}

\subsubsection*{NIP henselian valued fields of mixed characteristic}
We now assemble the lemmas from the previous subsections into a list of the complete theories of NIP henselian valued fields of mixed characteristic with residue field elementarily equivalent to $k$.
%
We first recall that the class of finitely ramified henselian valued fields satisfies the $\mathrm{AKE}^{\preceq}$ principle:

\begin{fact}\label{fact:AKE.preceq}
Let $(K,v) \subseteq (L,u)$ be an extension of henselian finitely ramified fields.
\begin{align*}
    \underbrace{(K,v) \preceq (L,w)}_{\text{in }\mathcal{L}_\mathrm{val}}
    \Longleftrightarrow
    \underbrace{vK\preceq wL}_{\text{in }\mathcal{L}_{\mathrm{oag}}}
    \text{ and }
    \underbrace{Kv\preceq Lw}_{\text{in }\mathcal{L}_{\mathrm{ring}}}.
\end{align*}
\end{fact}
\begin{proof}
See
\cite[Theorem 4.3.4]{ErshovMult}
or
\cite[Satz V.5~I)~iii)]{ZieglerDiss}).
\end{proof}

This fact now allows us to view henselian finitely ramified fields with value group
a $\mathbb{Z}$-group essentially (i.e., up to elementary equivalence)
as finite extensions of Cohen fields. 
\begin{lemma}\label{lem:finite_ramification}
Let $(K,v)$ be a henselian finitely ramified valued field of mixed characteristic with value group $vK\equiv\mathbb{Z}$.
Then $(K,v)$ is elementarily equivalent to a finite
extension $(L,u)$ of a Cohen field $(C(Kv),w)$, where $w$ denotes the unique 
non-trivial henselian valuation on $C(Kv)$ of mixed characteristic, and $u$ is its unique prolongation to $L$.
\end{lemma}
\begin{proof}
Let $(K^*,v^*) \succ (K, v)$ be an
$\aleph_1$-saturated elementary extension.
Then $(K^*,v^*)$ is also finitely ramified with $v^{*}K^{*}\equiv\mathbb{Z}$.
Consider the 
standard decomposition
\Kstarstandarddecompositionfr
of $(K^*,v^*)$.
Since $v^{*}_{0}$ is henselian,
we may apply
\cite[Theorem 2.9]{vdD},
to choose a section $\phi:K^{*}v^{*}_{0}\longrightarrow K^{*}$ of the residue map $\res_{v^{*}_{0}}$ of $v^{*}_{0}$;
this is an embedding of fields such that 
$\res_{v^{*}_{0}}\circ\phi$ is the identity.
In fact $\phi$ is an embedding of valued fields $(K^{*}v^{*}_{0},\bar{v}^{*})\longrightarrow(K^{*},v^{*})$
such that
the extension of residue fields is trivial
and the extension of value groups (with the value of $p$ distinguished) is elementary.
Applying Fact~\ref{fact:AKE.preceq},
we have $(K^{*}v_{0}^{*},\bar{v}^{*})\preceq(K^{*},v^{*})$.
By saturation, $(K^*v_0^*, \bar{v}^*)$ is complete with value group $\Delta_0^* \cong \mathbb{Z}$.
By Fact \ref{fact:Warner}, $(K^*v_0^*, \bar{v^*})$ is a finite extension of a Cohen subfield
$(C(K^*v^*),w^{*})$.
By \cite[Theorem 1.2]{AJ2},
and since $Kv \equiv K^*v^*$ we have 
\begin{align*}
	(C(K^*v^*),w^{*}) \equiv (C(Kv),w),
\end{align*}
where $(C(Kv),w)$ is a Cohen field over $Kv$.
Since $(C(K^{*}v^{*}),w^{*})$ and $(C(Kv),w)$ are elementarily equivalent, the latter admits a finite extension $(L,u)$ to which $(K^{*}v_{0}^{*},\bar{v}^{*})$ is elementarily equivalent.
Putting together this chain of elementary equivalences, this shows that $(K,v)$ is elementarily equivalent to a finite extension of $(C(Kv),w)$.
\end{proof}

We are now in a position to prove the main result of this section:
\begin{proposition}\label{prp:elem}
	Let $T_{k}=Th(k)$ be a complete 
$\Lring$-theory of NIP fields of characteristic $p$.
Let $(K,v)$ be a NIP henselian valued field of mixed characteristic with residue field elementarily equivalent to $k$. Then, each of the following holds:
	\begin{enumerate}[leftmargin=1.0cm]
		\item[{\Shelahcasea}] If $k$ is finite, then $(K,v)$ is $\Lval$-elementarily equivalent to
		a finite extension of a model of $T_{0}(\mathbb{F}_p,\rmGamma,\gamma)$,
		for some $(\rmGamma, \gamma)$ such that $\gamma$ is the minimum
		positive element of $\rmGamma$.\\
		In other words: if $k$ is finite, then
		$(K,v)$ is elementarily equivalent to a (generalized) power series field
		over a finite extension of the $p$-adics where $v$ corresponds to the
		composition of the $p$-adic valuation and the power series valuation.
		\item[{\Shelahcaseb}]If $k$ has imperfection degree $e\in\mathbb{N}_{>0}\cup\{\infty\}$, then $(K,v)$ is
		$\Lval$-elementarily equivalent to
		a finite extension of a model of $T_{e}(k,\rmGamma,\gamma)$,
		for some $(\rmGamma, \gamma)$ such that 
		$\gamma$ is the minimum positive element of
		$\rmGamma$.\\ In other words: if $k$ is imperfect,
		then $(K,v)$ is elementarily equivalent to a (generalized) power series field
		over a finite extension of the Cohen field $C(k)$ where $v$ corresponds to the
		composition of the Cohen valuation and the power series valuation.
		\item[{\Shelahcasec}] If $k$ is perfect and infinite, then $(K,v)$ is either
        	\begin{enumerate}[leftmargin=1.0cm]
        		\item[{\Shelahcaseci}] elementarily equivalent to a finite extension
    			of a model of $T_{e}(k,\rmGamma, \gamma)$, 
    			such that
    			$e \in \mathbb{N}\cup \{\infty\}$, the image of
    			$\gamma$ in $\rmGamma/\rmGamma_{\gamma -}$ is minimum
    			positive and $\rmGamma_{\gamma -}$ is $p$-divisible,
    			
\listintertext{or}

        		\item[{\Shelahcasecii}] a model of 
			$T_{}(k,\rmGamma, \gamma, F,v_{F})$, such that $\rmGamma/\rmGamma_{\gamma -}$ is not 
			discrete, $[-\gamma, \gamma] \subseteq p\rmGamma$, and $(F,v_{F})$ is compatible with $(\rmGamma,\gamma)$;
i.e.,
			$(K,v)$ is a model of any completion of $T_{}(k, \rmGamma, \gamma)$.
	\end{enumerate}
\end{enumerate}
\end{proposition}
\begin{proof}
Let $k$ be a NIP field, and $(K,v)$ 
NIP henselian valued field with $Kv \equiv k$. 
Once more, we consider the standard decomposition of $(K,v)$:
\Kstandarddecomposition

\begin{enumerate}[leftmargin=1.0cm]
	\item[{\Shelahcasea}]
	If $k$ is finite, then we have $Kv=k$ and 
	Remark \ref{rem:3.6}
	implies that $(K,v)$ is finitely ramified, so 
	$\rmDelta_p=\{0\}$ and $vK$ has a minimum positive element.
	Thus, $(Kv_0, \bar{v})$ is a henselian valued field of mixed
	characteristic with value group 
	$\mathbb{Z}$ and finite residue field. Applying \cite[Theorem 3.1]{PrR},
	it is elementarily equivalent
	to a finite extension $L$ of the $p$-adics $\mathbb{Q}_{p}$, equipped with the unique extension $w$ of the $p$-adic valuation.
	We now use $G$ to denote the ordered abelian group $vK/\rmDelta_0$.
	By the Ax-Kochen/Ershov Theorem for equicharacteristic $0$ 
	(\cite[Theorem 5.11]{vdD}), 
	$(K,v_0)$ is elementarily equivalent to
	the generalized power series field 
	$(L((G)),v_{G})$.
	Moreover, since $v$ is finitely ramified, it is $\emptyset$-definable
	in the language of rings. Thus, we have 
	$$(K,v) \equiv (L((G)), w \circ v_G),$$
	The latter is a finite extension of $(\mathbb{Q}_p((G)), w\circ v_{G})$.
	The valued field $(\mathbb{Q}_p((G)), w \circ v_{G})$
	a model of $T_{0}(\mathbb{F}_p,\rmGamma,\gamma)$,
	where $\rmGamma$ is the value group of $w\circ v_{G}$ on $\mathbb{Q}_p((G))$ 
	with minimum positive element $\gamma$.

	\item[{\Shelahcaseb}]
	If $k$ has imperfection degree $e\in\mathbb{N}_{>0}\cup\{\infty\}$, then it is not the residue field of
	any Kaplansky valued field. Thus, Theorem \ref{thm:1} implies
	that $(Kv_p, \bar{v})$ is trivially valued and $(K,v_p)$ is finitely ramified.
	In particular, $vK$ has a minimum positive element $\gamma_0$ and
	$(Kv_0, \bar{v})$ is a henselian valued field of 
	mixed characteristic with value group $\mathbb{Z}$ and residue field $k$.
	By Lemma \ref{lem:finite_ramification},
	$(Kv_{0},\bar{v})$ is elementarily equivalent to a finite extension $(L,w)$ of the Cohen field $(C(k),w)$.
	Again, writing $G=vK/\rmDelta_0$, we have $(K,v_0)\equiv (L((G)), v_G)$.
	Once more, the $\Lring$-definability of $v$ implies that
	$$(K,v) \equiv (L((G)), w\circ v_G)$$
	holds.
	Like in the previous case, we note that the 
	extension $$(C(k)((G)),w \circ v_G) \subseteq (L((G)), w \circ v_G)$$
	is finite and that $(C(k)((G)),w \circ v_G)$ is a model of 
	$T_{e}(k,\rmGamma,\gamma)$, where again $\rmGamma$ is the value group of $w\circ v_{G}$ on $C(k)((G))$ with minimum positive element $\gamma$.

	\item[{\Shelahcasec}] Assume that $k$ is infinite and perfect.
	If $(K,v)$ is a henselian valued field of mixed characteristic with residue
	field $Kv\equiv k$, then Theorem \ref{thm:1} implies that one of the following
	holds:
	\begin{enumerate}[leftmargin=1.0cm]
	\item[{\bf (i)}] $(K,v_p)$ is finitely ramified and $(Kv_p, \bar{v})$
	is trivial or separably defectless Kaplansky, i.e.~$(K,v)$ 
		satisfies clause {\maincaseb}, or
	\item[{\bf (ii)}] 
		$(Kv_0, \bar{v})$ is defectless Kaplansky, i.e.~$(K,v)$
	satisfies clause {\maincasec}.
	\end{enumerate}
	We show that cases {\bf(i)} and {\bf(ii)} correspond exactly to 
	{\Shelahcaseci} and {\Shelahcasecii} respectively.

	We first assume that $(K,v_p)$ is finitely ramified and $(Kv_p, \bar{v})$
	is trivial or separably defectless Kaplansky
	of imperfection degree $e$. If $(Kv_p, \bar{v})$ is trivial,
	we have $Kv_p \equiv k$ and hence $Kv_p$ is perfect and admits no Galois extensions
	of degree divisible by $p$. In particular, the trivially valued field 
	$(Kv_p, \bar{v})$ is separably defectless Kaplansky 
	(of imperfection degree $0$). Thus, we may treat these two subcases
	simultaneously.
	Now, let $(K^*,v^*) \succ (K,v)$ be an $\aleph_1$-saturated elementary extension.
	Consider the standard decomposition of $(K^*,v^*)$:
	\Kstarstandarddecomposition
	As $v_p$ is finitely ramified, it is $\emptyset$-definable, and hence we
	also have $(K^*,v^*_p) \succ (K,v_p)$ and $(K^*v^*_p, \bar{v}^*) \succ (Kv_p, \bar{v})$. In particular, $(K^*v^*_p, \bar{v}^*)$ is separably algebraically
	maximal Kaplansky. By saturation, $(K^*v_0^*, \bar{v}_p^*)$ is a complete
	mixed characteristic valued field with value group isomorphic to $\mathbb{Z}$.
	By Fact \ref{fact:Warner},
	$(K^*v_0^*, \bar{v}_p^*)$ is a finite extension of a Cohen subfield $(C(K^*v_p^*),w)$.
	Hence, we get a
	finite extension $$(C(K^*v_p^*), \bar{v}^* \circ w) \subseteq (K^*v^*_0, \bar{v}^*
	\circ \bar{v}^*_p)$$
	of valued fields. 
	This gives rise to a finite extension 
	$$(C(K^*v_p^*)((G)), \bar{v}^* \circ w \circ v_G) \subseteq (K^*v^*_0 ((G)), \bar{v}^*
	\circ \bar{v}^*_p \circ v_G)$$
	of valued fields, where $G=v^*K^*/\Delta_0^*$ and $v_G$ denotes the corresponding power series
	valuation in both cases.
	Since $\bar{v}^*$ is separably algebraically
	maximal Kaplansky of imperfection degree $e$ on $K^*v_p^*$,
	$(C(K^*v_p^*)((G)), u)$ is a model of $T_{e}(k,\rmGamma, \gamma)$,
	where $u=\bar{v}^* \circ w \circ v_G$ and $\rmGamma$ denotes the value group of $u$ on $C(K^*v_p^*)((G))$ and $\gamma=u(p)$.  
	Since the Ax-Kochen/Ershov principle in equicharacteristic $0$ still
	holds if one adds structure on the (purely stably embedded) residue field
	(i.e., $K^*v^*_0$),
	we have 
	$$(K^*,v^*) \equiv (K^*v^*_0 ((G)), \bar{v}^*
	\circ \bar{v}^*_p \circ v_G).$$
	Thus, $(K,v)$ is indeed elementarily equivalent to a finite extension
	of a model of $T_e(k,\rmGamma,\gamma)$ as desired.

	Finally, if $(Kv_0, \bar{v})$ is defectless Kaplansky, then 
	in particular $(K,v)$ is algebraically maximal.
	Therefore,  $(K,v)$ is a model of $T(k,vK,v(p))$.
	Moreover, $\rmDelta_0$ is $p$-divisible,
	and so the inclusion $[-v(p), v(p)]\subseteq p\cdot vK$ holds and $vK/\rmDelta_p$ is not discrete. 
	\qedhere
\end{enumerate}
\end{proof}

\section{{A refinement of Conjecture \ref{sh}}}

The last result of this paper is a reformulation of Conjecture \ref{sh}.
Theorem \ref{thm:3} below gives a list of first-order theories of valued henselian fields such that, if Conjecture \ref{sh} holds, 
any NIP field $K$ admits a henselian valuation $v$ such that 
$(K,v)$ is a model of one of the theories on the list.
Since all the theories appearing in Theorem \ref{thm:3} are either complete or their completions can easily be described (cf.~Remark \ref{rem:complete}) and
moreover (by Theorem \ref{thm:1}) NIP, Theorem \ref{thm:3}
gives a converse of sorts for Conjecture \ref{sh}.

Recall from the previous section that for an ordered abelian group $\Gamma$ and $\gamma \in \Gamma$, we use $\Gamma_{\gamma-}$ to denote the maximal convex subgroup not containing $\gamma$ and $\Gamma_{\gamma+}$
to denote the minimal convex subgroup containing $\gamma$.

\begin{theorem}\label{thm:3}
Suppose that Conjecture \ref{sh} holds.
If a field $K$ is NIP then it is finite or
admits a henselian valuation $v$,
such that $(K,v)$ is
    \begin{enumerate}[label={\textbf{\textsc{(\roman*)}}}]
        \item
        a model of~$T(\mathbb{C},\rmGamma)$, or equivalently, $(K,v)\equiv(\mathbb{C}((\Gamma)),v_{\rmGamma})$,
        \item
        a model of
        $T(\mathbb{R},\rmGamma)$, or equivalently, $(K,v)\equiv(\mathbb{R}((\Gamma)),v_{\rmGamma})$,
        \item
        a model of
        $T^{\mathrm{sd}}_{e}(\widetilde{\mathbb{F}}_{p},\rmGamma)$,
        for $e\in\mathbb{N}\cup\{\infty\}$,
        and where
        $\rmGamma$ is $p$-divisible. In particular, in case $K$ is perfect,
        we have $(K,v) \equiv (\widetilde{\mathbb{F}_p}((\Gamma)),v_{\rmGamma})$,
        \item
        elementarily equivalent to a finite extension of a model of
        $T_{0}(\mathbb{F}_{p},\rmGamma,\gamma)$,
        where $\gamma$ is minimum positive in $\rmGamma$,
        or equivalently, $(K,v) \equiv (L((\Delta)), w \circ v_{\Delta})$
        where $\Delta=\Gamma/\Gamma_{\gamma+}$ and $(L,w)$ is a finite extension of $(\mathbb{Q}_p,w)$,
        \item
        elementarily equivalent to a finite extension of a model of
         $T_{e}(\widetilde{\mathbb{F}}_{p},\rmGamma,\gamma)$,
        where the image of $\gamma$ is minimum positive in $\rmGamma/\rmGamma_{\gamma-}$,
        and $\rmGamma_{\gamma-}$ is $p$-divisible, or equivalently,
        $(K,v) \equiv (L, \nu \circ w)$
        where $(L,w)$ is a 
        finitely ramified\footnote{Recall that our definition of finitely ramified does not require the value group to have rank $1$, cf.~p.~2.} henselian valued field with value group 
        $\rmGamma/\rmGamma_{\gamma-}$ and
        with residue field $k$ and such that
        $(k,\nu) \models  T^{\mathrm{sd}}_{e}(\widetilde{\mathbb{F}}_{p},\rmGamma_{\gamma-})$,
        \item
        a model of 
        $T^{}(\widetilde{\mathbb{F}}_{p},\rmGamma,\gamma)$,
        where
        $\rmGamma_{\gamma+}$ is $p$-divisible.
    \end{enumerate}
\end{theorem}

Before we prove Theorem \ref{thm:3}, we comment on the extent to which
the theories occurring in the statement are complete and, if not, how to complete them:
\begin{remark}
It follows from Theorem \ref{thm:1} that all of the valued fields in cases
{\textbf{\textsc{(i)}}}--{\textbf{\textsc{(vi)}}}
are NIP. \label{rem:complete}
In particular, assuming Shelah's conjecture, the list
{\textbf{\textsc{(i)}}}--{\textbf{\textsc{(vi)}}}
in Theorem \ref{thm:3} gives a classification of theories of NIP fields.

Each theory appearing in cases \textbf{\textsc{(i)}}--\textbf{\textsc{(v)}} is complete.
Case {\textbf{\textsc{(vi)}}} of the previous theorem deals with the theories $T(\widetilde{\mathbb{F}}_{p},\rmGamma,\gamma)$, where $(\rmGamma,\gamma)$ is such that $\rmGamma_{\gamma+}$ is $p$-divisible.
By Lemma \ref{lem:Ershov.Ziegler}, the completions of such theories are exactly the theories $T(\widetilde{\mathbb{F}}_{p},\rmGamma,\gamma,F,v_{F})$, where $(F,v_{F})$ is compatible with $(\rmGamma,\gamma)$.
The class of valued fields $(F,v_{F})$ compatible with some such $(\rmGamma,\gamma)$ admits a simple algebraic description.
Namely, it is the class of tame valued fields $(F,v_{F})$ which are algebraic extensions of the maximal unramified extension $\mathbb{Q}_{p,\mathrm{alg}}^{\mathrm{unram}}$ of $\mathbb{Q}_{p,\mathrm{alg}}$, equipped with the unique extension of the $p$-adic valuation.
In this way, Theorem \ref{thm:3} even provides a conjectural classification of the complete theories of NIP fields.
\end{remark}

\begin{proof}[Proof of Theorem \ref{thm:3}]
Assume that $K$ is an infinite NIP field
which is neither separably closed nor real closed.
By Conjecture \ref{sh},
$K$ admits a non-trivial henselian valuation $v$.
Without loss of generality, we may assume that $v$ is the canonical henselian valuation on $K$.
By
\cite[Theorem B]{Jah19}%
\footnote{Although in this proof we reference \cite{Jah19} which in turn references the present paper, the argument is not circular since \cite{Jah19} only refers to results in Section 4 of the present paper.}%
,
$(K,v)$ is NIP,
and in particular the residue field $k:=Kv$ is NIP.
Since $v$ is the canonical henselian valuation, $k$ is either separably closed or not henselian.
Since $k$ is also NIP, applying Conjecture \ref{sh} once more
yields that $k$ is either separably closed, real closed, or finite.

Applying Theorem \ref{thm:1}, $(K,v)$ lies in case {\maincasea}, {\maincaseb}, or {\maincasec} of that theorem.
If $\mathrm{char}(k)=0$, then $(K,v)$ lies in case {\maincasea}
and we must have that $k$ is algebraically closed or real closed,
i.e.~elementarily equivalent to either $\mathbb{C}$ or $\mathbb{R}$.
Note that in neither case is the value group $\rmGamma:=vK$ divisible, since otherwise $K$ is already algebraically closed or real closed, contrary to our assumption.
If $k$ is algebraically closed, $(K,v)$ is a model of $T(\mathbb{C}, \rmGamma)$, so
{\textbf{\textsc{(i)}}}
holds. 
If $k$ is real closed, $(K,v)$ is a model of $T(\mathbb{R},\rmGamma)$, so
{\textbf{\textsc{(ii)}}}
holds.

Next, if $\mathrm{char}(K)=p>0$, then $(K,v)$ again lies in case {\maincasea} from \ref{thm:1},
and we must have that $k$ is separably closed or finite.
Moreover, since $v$ is non-trivial, {\maincasea} implies that
$(K,v)$ is separably defectless Kaplansky. In particular,
$\rmGamma$ is $p$-divisible and
$k$ is perfect and Artin-Schreier closed,
and so $k$ is algebraically closed.
It follows that $(K,v)$ is a model of
$T_{e}^{\mathrm{sd}}(\widetilde{\mathbb{F}}_{p},\rmGamma)$, where $e$ denotes the imperfection degree of $K$,
and hence
{\textbf{\textsc{(iii)}}}
holds.

We now turn to the case that $(K,v)$ is of mixed characteristic $(0,p)$, with $p>0$.
In this case, $(K,v)$ lies in case {\maincaseb} or case {\maincasec} from \ref{thm:1}.
Also, $k$ is either separably closed or finite.
If $k$ is finite, then
by Proposition \ref{prp:elem}{\Shelahcasea},
$(K,v)$ is $\Lval$-elementarily equivalent to a finite extension of a model of
$T_{0}(\mathbb{F}_{p},\rmGamma,\gamma)$,
for some pair $(\rmGamma,\gamma)$ such that $\gamma$ is minimum positive in $\rmGamma$.
This verifies
{\textbf{\textsc{(iv)}}}.
Suppose now that $k$ is separably closed.
Let $w$ denote a finest valuation on $k$, i.e.~one which admits no proper refinement.
Such valuations always exist, 
and in this case $w$ is henselian and the residue field $kw$ is $\widetilde{\mathbb{F}}_{p}$.
By Proposition \ref{prp:elem}{\Shelahcasec},
there is a dichotomy between cases
{\Shelahcaseci}
and
{\Shelahcasecii}.
In case 
{\Shelahcaseci},
$(K,v)$ is $\Lval$-elementarily equivalent to a finite extension of a model of
$T_{e}(\widetilde{\mathbb{F}}_{p},\rmGamma,\gamma)$
such that $e\in\mathbb{N}\cup\{\infty\}$
and the image of $\gamma$ in $\rmGamma/\rmGamma_{\gamma-}$ is minimum positive,
and $\rmGamma_{\gamma-}$ is $p$-divisible.
This verifies
{\textbf{\textsc{(v)}}}.
In case 
{\Shelahcasecii},
$(K,v)$ is a model of
$T_{}(\widetilde{\mathbb{F}}_{p},\rmGamma,\gamma)$
and $[-\gamma,\gamma]\subseteq p\rmGamma$, which implies that $\rmGamma_{\gamma+}$ is $p$-divisible.
This verifies
{\textbf{\textsc{(vi)}}}.
\end{proof}

\section*{Acknowledgements}
We would like to thank Philip Dittmann, Yatir Halevi and Assaf Hasson for various
helpful discussions. Many ideas contained in this article came out of discussions
during the \emph{Model Theory,
Combinatorics and Valued fields} trimester at the 
Institut Henri Poincar\'e
and the Centre \'{E}mile Borel,
and we would like to extend our thanks to the organisers.
Sylvy Anscombe was also partially supported by The Leverhulme Trust under grant RPG--2017--179 and by a visiting scholarship at St John's College, Oxford. 
Franziska Jahnke was also funded by the Deutsche Forschungsgemeinschaft 
(DFG, German Research Foundation) under Germany’s Excellence Strategy
EXC 2044--390685587, Mathematics M\"unster: Dynamics--Geometry--Structure and
the CRC878, as well as by a Fellowship from the Daimler and Benz Foundation.
Last but not least, we would like to thank the anonymous referee for numerous helpful comments 
and suggestions.

\def\bibfont{\footnotesize}
\bibliographystyle{plain}

\end{document}